\documentclass[11pt]{article}

\usepackage{geometry}
\usepackage{amssymb}
\usepackage{amsmath}
\usepackage{amsthm}
\usepackage{amsopn}
\usepackage{amscd}
\usepackage{exscale}
\usepackage{graphicx}
\usepackage{booktabs}
\usepackage[all,cmtip]{xy}

\theoremstyle{plain}
\newtheorem{thm}{Theorem}[section]
\newtheorem{prop}[thm]{Proposition}

\newtheorem{cor}[thm]{Corollary}

\newtheorem{lem}[thm]{Lemma}

\theoremstyle{definition}

\theoremstyle{remark}

\newtheorem{example}[thm]{Example}

\newcommand{\QQ}{{\bar{\mathbb{Q}}}}

\newcommand{\Q}{\mathbb{Q}}

\newcommand{\Z}{\mathbb{Z}}

\author{Manisha Kulkarni, Dipramit Majumdar \& Balasubramanian Sury}
\title{$l$-Class groups of cyclic extensions of prime degree $l$}
\date{}

\begin{document}
\maketitle

\begin{abstract}
Let $K/F$ be a cyclic extension of prime degree $l$ over a number
field $F$. If $F$ has class number coprime to $l$, we study the
structure of the $l$-Sylow subgroup of the class group of $K$. In
particular, when $F$ contains the $l$-th roots of unity, we obtain
bounds for the $\mathbb{F}_l$-{\rm rank} of the $l$-Sylow subgroup
of $K$ using genus theory. We obtain some results valid for general
$l$. Following that, we obtain more complete results for $l=5$ and
$F = \mathbb{Q}(\zeta_5)$. The {\rm rank} of the $5$-class group of
$K$ is expressed in terms of power residue symbols. We compare our
results with tables obtained using SAGE (the latter is under GRH).
We obtain explicit results in several cases. Using these results,
and duality theory, we deduce results on the $5$-class numbers of
fields of the form $\mathbb{Q}(n^{1/5})$. \footnote{Mathematics
Subject Classification 11 R 29, 13 C 20}
\end{abstract}

\tableofcontents

\section{Introduction}
\noindent We study the $l$-class group of $K$, where $K$ is a cyclic
extension of degree $l$ over a number field $F$ which contains the
$l$-th roots of unity and has trivial $l$-class group, where $l$ is
an odd prime. Denote by $\tau$ a generator of Gal $(K/F)$. The
$l$-class group $S_K$ is a $\mathbb{Z}_l[\zeta_l]$-module since
$$\mathbb{Z}_l[\zeta_l] \cong \mathbb{Z}_l[Gal(K/F)]/(1+ \tau +
\cdots + \tau^{l-1})$$ where $\zeta$ corresponds to $\tau$. As a
module over the discrete valuation ring $\mathbb{Z}_l[\zeta]$ whose
maximal ideal is generated by $\lambda = 1- \zeta$, the $l$-class
group $S_K$ of $K$ decomposes as
$$S_K \cong \mathbb{Z}_l[\zeta_l]/(\lambda^{e_1}) \oplus
\mathbb{Z}_l[\zeta_l]/(\lambda^{e_2}) \oplus \cdots \oplus
\mathbb{Z}_l[\zeta_l]/(\lambda^{e_t})$$ for some $1 \leq e_1 \leq
e_2 \leq \cdots \leq e_t$. Our goal is to compute the rank of $S_K$
which is the dimension of the $\mathbb{F}_l$-vector space $S_K
\otimes_{\mathbb{Z}_l} \mathbb{F}_l$. To find the $e_i$'s, one looks
at
$$s_i = |\{e_j : e_j = i \}|.$$ Then, the rank of
the $\mathbb{Z}_l[\zeta_l]$-module $\lambda^{i-1}S_K/\lambda^i S_K$
is $t-s_1- \cdots-s_{i-1}$ - this is also called the
$\lambda^i$-rank of $S_K$. To compute these numbers, we consider the
decreasing filtration
$$S_K \supset \lambda S_K \supset \lambda^2 S_K \supset \cdots$$
and construct ideal classes generating the pieces
$\lambda^{i-1}S_K/\lambda^i S_K$ and construct genus fields
corresponding to them. This is difficult to carry out explicitly in
general. However, the general analysis does lead to expressions and
bounds for the rank of $S_K$ such as:\\
{\bf Proposition.} $S_{K}$ is isomorphic to the direct product of an
elementary abelian $\ell$-group of {\rm rank} $s_1$ and an abelian
$\ell$ group of {\rm rank} $$(\ell -1)(t-s_{1}) - (\ell -3)s_{2}-
(\ell -4)s_{3}- \cdots -s_{\ell -2}.$$ In particular, $${\rm rank}
S_K = (\ell -1)t - (\ell -2)s_{1}- (\ell -3)s_{2}- \cdots -s_{\ell
-2}$$ satisfies the bounds
$$2t-s_1 \leq rank~S_K \leq (l-1)t-(l-2)s_1$$
both of which are achievable.\\
This is proved in section 3; the expression for the rank is almost
immediate but the direct sum decomposition proved in the proposition
is used while constructing the genus fields explicitly later. \vskip
3mm

\noindent In section 4, we assume that $F$ contains the $l$-th roots
of unity and construct genus fields corresponding to the pieces of
the class group as above. These fields are of the form $K(x_1^{1/l},
x_2^{1/l}, \cdots, x_t^{1/l})$. For a basis $\{P_j \}$ of ideal
classes for a piece, using Kummer theory to map the Galois group of
the corresponding genus field to $\mathbb{F}_l^t$, one writes down a
matrix with entries in $\mathbb{F}_l$ from that part of the class
group. This allows us to express the rank of that piece of the class
group in terms of the rank of a matrix of Artin symbols of the form
$\bigg( \frac{K(x_i^{1/l})/K}{P_j} \bigg)$ (see theorems 4.1, 4.2).
\vskip 3mm

\noindent In section 5, we specialize to $l=5$ and $F =
\mathbb{Q}(\zeta)$ which allows us to precisely work out the
previous results. The major part of the paper is contained in
sections 5 and 6. In section 5, we use ideles to rewrite the earlier
computations of the $s_i$'s in terms of Artin symbols in a more
explicit form in terms of local Hilbert symbols.
One of the results in section 5 is:\\
{\bf Theorem.} Let $K=F(x^{\frac{1}{5}})$, $x= u
\lambda^{e_{\lambda}} \pi_{1}^{e_{1}} \cdots \pi_{g}^{e_{g}}$ and $F
= \mathbb{Q}(\zeta)$ where each $\pi_i$ is a prime element congruent
to a rational integer modulo $5 \mathbb{Z}[\zeta]$ and $u$ is a unit
in $F$. Let
$M_{1}=K(x_{1}^{\frac{1}{5}},\cdots,x_{t}^{\frac{1}{5}})$ denote the
genus field of $K/F$, where $[M_{1}:K]=5^{t}$, $x_{i} \in F$ for
$1\leq i \leq t$, and $x_{i} \equiv \pm1,\pm7 \pmod{ \lambda^{5}}$.
For $1\leq i \leq t, 1\leq j \leq g,$ let $\nu_{ij}$ denote the
degree $5$ Hilbert symbol $\Big(\frac{x_{i},x}{(\pi_{j})}\Big)$ in
the local field $K_{\pi_j}$. Further, suppose
$$\nu_{i,g+1} = \Big(\frac{x_{i},\lambda}{(\lambda)}\Big)
\text {   for } 1\leq i \leq t, \text{ if the ideal} (\lambda)
\text{of $F$ ramifies in }K.$$ If $\gamma_{ij}\in \mathbb{F}_{\ell}$
are defined by the power symbol $\zeta^{\gamma_{ij}} =
(x_{i}^{\frac{1}{\ell}})^{\nu_{ij}-1}$, and $C_{1}$ is the matrix
$(\gamma_{ij}), 1\leq i \leq t, 1\leq j \leq u=g$ or $g+1$, we have
$$s_{1} = {\rm rank} C_{1}.$$

\noindent The above result is under the assumption that ambiguous
ideals are strongly ambiguous; in the contrary case, we have a very
similar statement with a slightly bigger matrix (see theorem 5.9). \\
A similar result is proved for computing $s_i$'s for $i>1$ (see
theorem 5.10). Thus, we have some results on the $\lambda^i$-rank of
the $5$-class group for general $i$ and the results on
$\lambda^2$-rank are easily computable in many situations. \vskip
5mm

\noindent We give tables of class groups obtained by using the SAGE
program and compare our results in its light. Interestingly, after a
close inspection of the tables, we were able to guess the following
general results which we prove (theorems 5.12,5.13,5.14,5,16,5.17):
\vskip 3mm

\noindent {\bf Theorem.} Let $p$ be a prime number congruent to $-1
\pmod{5}$. Let $F=\mathbb{Q}(\zeta_{5})$ and $K=F(p^{\frac{1}{5}})$.
Assuming that each ambiguous ideal class is strongly ambiguous, we
have that $25$ divides the class number of $K$. More precisely, the
$\lambda^2$-rank (to be defined below) of the $5$-class group
$S_{K}$ is $1$ and, $ 2 \leq \mathrm{rank } S_{K} \leq 4$. \vskip
3mm

\noindent {\bf Theorem.} Let $p$ be a prime number congruent to $\pm
7 \pmod{25}$ and $q$ be a prime number congruent to $-1 \pmod{5}$.
Let $F=\mathbb{Q}(\zeta_{5})$ and $K=F((pq)^{\frac{1}{5}})$.
[Assuming that each ambiguous ideal class is strongly ambiguous,] we
have that $125$ divides the class number of $K$. More precisely,
$\lambda^2$-rank of $S_{K}$ is $1$ and we have, $3 \leq \mathrm{rank
} S_{K} \leq 5$. \vskip 3mm

\noindent {\bf Theorem.} Let $p_{i} \equiv \pm 7 \pmod{25}$ for $1
\leq i \leq r$ be primes and $r \geq 2$. Let $n=p_{1}^{a_{1}} \cdots
p_{r}^{a_{r}}$, where $1 \leq a_{i} \leq 4$ for $1 \leq i \leq r$.
Let $F=\mathbb{Q}(\zeta_{5})$ and $K=F(n^{\frac{1}{5}})$. \\
(i) If all ambiguous ideal classes of $K/F$ are strongly ambiguous,
then the $\lambda^2$-rank
of $S_{K}$ is $r-1$ and $2r-2 \leq \mathrm{rank } S_{K} \leq 4r-4$. \\
(ii) If there are ambiguous ideal classes which are not strongly
ambiguous, then $s_{1} \leq 2$, $\lambda^{2}$-rank of $S_{K}$ is
greater than or equal to $r-3$ and $\max(2r-4,r-1) \leq \mathrm{rank
}~ S_{K} \leq 4r-4$. \vskip 3mm

\noindent {\bf Theorem.}  Let $p_{i} \equiv \pm 7 \pmod{25}$ for $1
\leq i \leq r$ be primes and let $q_{j}$ be primes such that $q_{j}
\equiv \pm 2 \pmod{5}$ but $q_{j} \not\equiv \pm7 \pmod{25}$ for $1
\leq j \leq s$. Let $n=p_{1}^{a_{1}} \cdots
p_{r}^{a_{r}}q_{1}^{b_{1}}\cdots q_{s}^{b_{s}}$, where $1 \leq
a_{i},b_{j} \leq 4$ for $1 \leq i \leq r$ and $1 \leq j \leq s$. Let
$n \not\equiv \pm 1, \pm 7 \pmod{25}$. Let $F=\mathbb{Q}(\zeta_{5})$
and
$K=F(n^{\frac{1}{5}})$. \\
(i) If all ambiguous ideal classes of $K/F$ are strongly ambiguous,
then the $\lambda^2$-rank of $S_{K}$ is $r+s-1$ and $2r+2s-2 \leq
\mathrm{rank } S_{K} \leq 4r+4s-4$.\\
(ii) If there are ambiguous ideal classes which are not strongly
ambiguous, then $s_{1} \leq 1$, $\lambda^{2}$-rank of $S_{K}$ is
greater than or equal to $r+s-2$ and $\max(2r+2s-3,r+s-1) \leq
\mathrm{rank } S_{K} \leq 4r+4s-4$. \vskip 3mm

\noindent {\bf Theorem.} Let $p_{i} \equiv \pm 7 \pmod{25}$ for $1
\leq i \leq r$ be primes
 and let $q_{j}$ be primes such that $q_{j} \equiv \pm 2 \pmod{5}$
  but $q_{j} \not\equiv \pm7 \pmod{25}$ for $1 \leq j \leq s$ with $s \geq 2$.
   Let $n=p_{1}^{a_{1}} \cdots p_{r}^{a_{r}}q_{1}^{b_{1}}\cdots q_{s}^{b_{s}}$,
    where $1 \leq a_{i},b_{j} \leq 4$ for $1 \leq i \leq r$ and $1 \leq j \leq s$.
    Let $n \equiv \pm 1$ or $\equiv \pm 7 \pmod{25}$. Let $F=\mathbb{Q}(\zeta_{5})$ and
    $K=F(n^{\frac{1}{5}})$.\\
(i) If all ambiguous ideal classes of $K/F$ are strongly ambiguous,
then the
$\lambda^2$-rank of $S_{K}$ is $r+s-2$ and $2r+2s-4 \leq \mathrm{rank } S_{K} \leq 4r+4s-8$.\\
(ii) If there are ambiguous ideal classes which are not strongly
ambiguous, then $s_{1} \leq 1$, $\lambda^{2}$-rank of $S_{K}$ is
greater than or equal to $r+s-3$ and $\max(2r+2s-5,r+s-2) \leq \mathrm{rank }
S_{K} \leq 4r+4s-8$. \vskip 3mm

\noindent From the last three theorems, one may deduce information
about certain $5$-class groups of purely quintic extensions of
$\mathbb{Q}$ such as corollary 6.7: \vskip 3mm

\noindent  {\bf Corollary:}\\
Let $N$ be a positive integer of one of the following forms. Then,
the $5$-class group of $L=\mathbb{Q}(N^{\frac{1}{5}})$ is either
trivial or cyclic:

\begin{itemize}
\item Let $N=p^{a}$, where $p \equiv \pm2 \pmod{5}$ is a prime, $1 \leq a \leq 4$.
\item Let $N=q_{1}^{a_{1}}q_{2}^{a_{2}}$ where  $q_{i} \equiv \pm 2 \pmod{5}$ but $q_{i} \not\equiv \pm 7 \pmod{25}$, $1 \leq a_{i} \leq 4$ for $i=1,2$ such that $N \equiv \pm 1, \pm 7 \pmod{25}$.
\item Let $N=p^{a}$, where $p \equiv -1 \pmod{5}$ is a prime, $1 \leq a \leq 4$.
\item Let $N=p_{1}^{a_{1}}p_{2}^{a_{2}}$ where $p_{i} \equiv \pm 7 \pmod{25}$, $1 \leq a_{i} \leq 4$ for $i=1,2$ such that $N \equiv \pm 1, \pm 7 \pmod{25}$.
\item $N=p^{a}q^{b}$ where $p \equiv \pm 7 \pmod{25}$ , $q \equiv \pm 2 \pmod{5}$ but $q \not\equiv \pm 7 \pmod{25}$ and $1 \leq a,b \leq 4$such that $N \not\equiv \pm 1, \pm 7 \pmod{25}$.
\item $N=q_{1}^{a_{1}}q_{2}^{a_{2}}$ where  $q_{i} \equiv \pm 2 \pmod{5}$ but $q_{i} \not\equiv \pm 7 \pmod{25}$, $1 \leq a_{i} \leq 4$ for $i=1,2$ such that $N \not\equiv \pm 1, \pm 7 \pmod{25}$.
\item  $N=p_{1}^{a_{1}}p_{2}^{a_{2}}q^{b}$ where $p_{i} \equiv \pm 7 \pmod{25}$, $q \equiv \pm 2 \pmod{5}$ but $q \not\equiv \pm 7 \pmod{25}$  $1 \leq a_{i},b \leq 4$ for $i=1,2$ such that $N \equiv \pm 1,\pm 7 \pmod{25}$.
\item $N=p^{a}q_{1}^{a_{1}}q_{2}^{a_{2}}$ where $p \equiv \pm 7 \pmod{25}$, $q_{i} \equiv \pm 2 \pmod{5}$ but $q_{i} \not\equiv \pm 7 \pmod{25}$, $1 \leq a,a_{i} \leq 4$ for $i=1,2$ such that $N \equiv \pm 1,\pm 7 \pmod{25}$.
\item  $N=q_{1}^{a_{1}}q_{2}^{a_{2}}q_{3}^{a_{3}}$ where  $q_{i} \equiv \pm 2 \pmod{5}$ but $q_{i} \not\equiv \pm 7 \pmod{25}$, $1 \leq a_{i} \leq 4$ for $i=1,2,3$, such that $N \equiv \pm 1, \pm 7 \pmod{25}$.
\item Let $N=p^{a}q^{b}$, where $p \equiv -1 \pmod{5}$ and $q \equiv \pm 7 \pmod{25}$ are primes, $1 \leq a,b \leq 4$.
\end{itemize}
\vskip 5mm

\noindent The above corollary is proved in section 6 where we
consider quintic fields $L = \mathbb{Q}(n^{1/5})$. If $F =
\mathbb{Q}(\zeta)$ as before, then $K = L(\zeta_5) =
\mathbb{Q}(n^{1/5},\zeta_5)$ has Galois group over $L$ to be cyclic
of order $4$, generated by $\sigma : \zeta \mapsto \zeta^3$. If
$\tau : n^{1/5} \mapsto \zeta n^{1/5}$ in Gal $(K/F)$, then
Gal$(K/\mathbb{Q})$ is the affine group on $\mathbb{F}_5$; viz.,
$<\sigma> \ltimes <\tau>$ where $\sigma \tau \sigma^{-1} = \tau^3$.
The group $S_K$ is a $\mathbb{Z}_5[G]$-module where $G =$
Gal$(K/L)$. For any $\mathbb{Z}_5[G]$-module $C$, one has a
decomposition
$$C = C^+ \oplus C^- \oplus C^{--}$$
where $C^+ = \{a \in C : {\sigma}a = a \}$, $C^- = \{a \in C :
{\sigma}a =- a \}$ and $C^{--} = \{a \in C : {\sigma^2}a =
-a \}$. Using this module structure and Kummer theory, we prove
in section 6 the following theorem:\\
{\bf Theorem.} If $L = \mathbb{Q}(n^{1/5})$ where $n$ is $5$-th
power free, then $rank~S_L \leq t-s_1+ rank~(S_K/(1- \zeta)S_K)^+
\leq 2t-s_1 \leq rank~S_K$.\\
Further, if $n = p_1^{\alpha_1} \cdots p_m^{\alpha_m}$ where the
primes $p_i \equiv \pm{2}$ or $\equiv -1$ modulo $5$, then\\
$rank~(S_K/(1-
\zeta)S_K)^+ =0$.\\

\noindent Our results along with computation in SAGE, sometimes
allows us to conclude about the existence of ambiguous ideal classes
which are not strongly ambiguous. For example, let $K=\Q(\zeta_{5},
\sqrt[5]{301}), L= \Q(\sqrt[5]{301})$. By theorem \ref{-7mod25}, if
all ambiguous ideal classes are strongly ambiguous, then, $2 \leq
\mathrm{rank} S_{K} \leq 4$. But using SAGE, we see that
$S_{K}=C_{5}$. Thus, in this situation we have ambiguous ideal
classes which are not strongly ambiguous and $t=s_{1}=1$. By
corollary \ref{t=s1} and theorem \ref{r=0}, we see that
$S_{L}=\{1\}$, which is confirmed by SAGE. \vskip 5mm

\noindent Historically, when $K = \mathbb{Q}(\sqrt{D})$ is a
quadratic field of discriminant $D$, Gauss's genus theory of
quadratic forms determines the {\rm rank} of the $2$-Sylow subgroup
of the ideal class group of $K$. C.S.Herz (\cite{Herz}) proved that
this {\rm rank} is $d-1$ or $d-2$ where $d = \omega(D)$, the number
of distinct prime divisors of $D$. In a series of papers (see
\cite{Gerth}, \cite{Gerth2}, \cite{Gerth3}), Frank Gerth III proved
several results on pure cubic extensions of $\mathbb{Q}$ and on
cyclic cubic extensions of $\mathbb{Q}$ and also obtained a
generalization of Herz's result for the $3$-Sylow subgroup of the
ideal class group of a cyclic extension of $\mathbb{Q}(\omega)$
where $\omega$ is a primitive $3$-rd root of unity. In two papers,
G.Gras (\cite{Gras1}, \cite{Gras2}) introduced and studied an
increasing filtration to obtain results on the narrow ideal class
group. As mentioned above, our results are proved using a decreasing
filtration and generalize Gerth's results to the case of any odd
prime $l$; they are more complete when $l=5$. \vskip 5mm

\noindent In the 1930's, Redei and Reichardt proved certain results
on class groups of some abelian extensions of $\mathbb{Q}$
(\cite{Redreich}). Curiously, the series of papers by Gerth do not
carry a reference to the old work of Redei and Reichardt.
Conversely, the newer papers which refer to Redei-Reichardt while
addressing similar questions (see, for instance, Greither-Kucera's
paper on the lifted root number conjecture \cite{Greikuc}), do not
seem to be aware of Gerth's work. \vskip 3mm

\noindent Redei matrices are certain square matrices which appeared
classically (see \cite{Redei}) and have been studied by others
(see \cite{Waterhouse},\cite{Kolster},\cite{Lu}) since
then. In our discussion, we construct similar matrices which are
rectangular in general. \vskip 5mm

\noindent Our results have a number of potential applications. For
example, by the class number formula, $\zeta_{K}^{\ast}(0) = \lim_{s
\to 0} s^{-r} \zeta_{K}(0)=- \frac{h_{K}R_{K}}{\omega_{K}}$, where
$r={\rm rank} O_K^{\ast}$, $\omega_{K}$ is the number of roots of
unity in $K$, $h_{K}$ is the class number of $K$ and $R_{K}$ is the
regulator of $K$. \vskip 2mm

\noindent Although Schanuel's conjecture predicts both $R_{K}$ and
$\zeta_{K}^{\ast}(0)$ are transcendental numbers, their quotient is
a rational number. \vskip 2mm

\noindent Number fields of the form
$K=\mathbb{Q}(\zeta_{5})(x^{\frac{1}{5}})$ are included in our study
and the results on $v_5(h_K)$ have potential use in obtaining
information on special values of Dedekind zeta functions. \vskip 2mm

\noindent In fact, for $K$ as above, one can check that $5||
\omega_{K}$ and thus the maximum power of $5$ that divides
$\frac{\zeta_{K}^{\ast}(0)}{R_{K}}$ is given by $v_{5}(h_{K})-1$.
\vskip 3mm

\noindent Another potential application of our results on the
$l$-class group of certain number fields is towards the existence of
$p$-descent for certain elliptic curves. \vskip 3mm

\noindent Indeed, assuming that the $\mathbb{F}_{2}$-rank of the
$4$-class group of $K=\mathbb{Q}(\sqrt{-2n})$ - where $n =
p_{0}p_{1}\cdots p_{k}$ is a product of distinct odd primes with
$p_{i} \equiv 1 \pmod{ 8}$ for $1 \leq i \leq k$ -  is  $0$ if $n
\equiv \pm 3 \pmod{8}$ and $1$ otherwise, Ye Tian showed
(\cite{Yeti}) that the elliptic curves $E^{(m)}/ \Q$ defined by
$my^{2}=x^{3}-x$, where $m = n$ or $2n$ such that $m \equiv 5, 6,$
or $7$ modulo $8$, have first $2$-descent and deduced the BSD
conjecture holds for these elliptic curves. \vskip 3mm

\noindent From our results on $5$-class numbers of fields of the
form $\mathbb{Q}(\zeta_5,n^{1/5})$, we use duality theory to deduce
results on the $5$-class number of the fields $\mathbb{Q}(n^{1/5})$
for some $n$. These have potential applications to the following
work of Calegari-Emerton on modular forms. Calegari and Emerton
showed (\cite{Calemer}) that if the class group of
$\mathbb{Q}(N^{\frac{1}{5}})$ is cyclic for a prime $N$, certain
local extensions of $\mathbb{Q}_5$ coming from normalized cuspidal
Hecke eigenforms are trivial. More precisely,: \vskip 3mm

\noindent Let $f$ be a normalized cuspidal Hecke eigenform of level
$N$. Let $K_{f}$ denote the extension of $\mathbb{Q}_{5}$ generated
by the $q$-expansion coefficients $a_{n}(f)$ of $f$. It is known
that $K_{f}$ is a finite extension of $\mathbb{Q}_{5}$. When $N$ is
a prime and $5 || (N-1)$, it is known due to Mazur that there exists
a unique (upto conjugation) weight $2$ normalized cuspidal Hecke
eigenform defined over $\mathbb{Q}_{5}$, satisfying the congruence
$$a_{l}(f) \equiv 1+l \pmod{\mathfrak{p}}$$
where $\mathfrak{p}$ is the maximal ideal of the ring of integer of
$K_{f}$. It is known that $K_{f}$ is a totally ramified extension of
$\mathbb{Q}_{5}$. Calegari and Emerton showed that if the class
group of $\mathbb{Q}(N^{\frac{1}{5}})$ is cyclic, then $K_{f} =
\mathbb{Q}_{5}$. \vskip 5mm

\section{Notations}
Let $\ell$ be an odd prime number. Let $F$ be a number field and
$K/F$ be a cyclic extension of degree $\ell$ over $F$. Let $C_{K}$
and $C_{F}$ denote the ideal class groups of $K$ and $F$
respectively. Let $S_{K}$ and $S_{F}$ denote their respective Sylow
$l$-subgroups which we sometimes refer to as the $l$-class groups.
The {\rm rank} of $S_{K}$ is defined to be the
$\mathbb{F}_{\ell}$-dimension of $S_{K} \otimes_{\mathbb{Z}_{\ell}}
\mathbb{F}_{\ell}$. \\

\noindent We have a natural action of ${\rm Gal}(K/F)$ on $C_{K}$
and on $S_K$. \\
We assume throughout that $S_F$ is trivial. It is convenient to use
additive notation. Denote by $\tau$ a generator of Gal $(K/F)$. Let
$\zeta_l$ be a fixed primitive $l$-th root of unity. The $l$-class
group $S_K$ is a $\mathbb{Z}_l[\zeta_l]$-module since
$$\mathbb{Z}_l[\zeta_l] \cong \mathbb{Z}_l[Gal(K/F)]/(1+ \tau +
\cdots + \tau^{l-1})$$ as the norm $1+ \tau + \cdots + \tau^{l-1}$
acts trivially on $S_K$. Denote the discrete valuation ring
$\mathbb{Z}_l[\zeta]$ by $R$; its maximal ideal is generated by
$\lambda = 1- \zeta$. As an $R$-module, the $l$-class group $S_K$ of
$K$ decomposes as
$$S_K \cong \mathbb{Z}_l[\zeta_l]/(\lambda^{e_1}) \oplus
\mathbb{Z}_l[\zeta_l]/(\lambda^{e_2}) \oplus \cdots \oplus
\mathbb{Z}_l[\zeta_l]/(\lambda^{e_t})$$ for some $1 \leq e_1 \leq
e_2 \leq \cdots \leq e_t$. Let
$$s_i := |\{e_j : e_j = i \}|$$
so that $t = s_1+s_2 + \cdots + s_{l-2}$ and $s_k =0$ for $k>l-2$
since $(\lambda^{l-1})=(l)$. We have a decreasing filtration
$$S_K \supset \lambda S_K \supset \lambda^2 S_K \supset \cdots$$
Denote by $S_K[\lambda]$, the kernel of multiplication by $\lambda$
on $S_K$; note that $t =$ rank $S_K[\lambda]$. Similarly, it is easy
to see that
$$s_i = {\rm rank} ((S_K[\lambda] \cap \lambda^{i-1}S_K) + \lambda^i
S_K)/\lambda^i S_K.$$ Also, the rank of the
$\mathbb{Z}_l[\zeta_l]$-module $\lambda^{i-1}S_K/\lambda^i S_K$ is
$t-s_1- \cdots-s_{i-1}$ which is called the $\lambda^i$-rank of
$S_K$. \vskip 5mm

\noindent By class field theory, the maximal abelian unramified
extension $M_0$ of $K$ satisfies $C_K \cong {\rm Gal}(M_0/K)$. The
genus field of $K/F$ is the maximal abelian extension $M$ of $F$
which is contained in $M_0$; then ${\rm Gal}(M/F)$ is the
abelianization of ${\rm Gal}(M_0/F)$. Moreover, $C_K/\lambda C_K
\cong {\rm Gal} (M/K)$ and is called the group of genera.

\noindent An ideal class $c$ in $C_{K}$ is said to be {\it
ambiguous} if $\tau c = c$; that is, if $c \in C_K[\lambda]$. Thus,
the subgroup $S_K[\lambda]$ of ambiguous ideal $l$-classes is an
elementary abelian $l$-group whose rank is that of $S_K/\lambda S_K$
(which we have denoted by $t$ above). The rank $t$ is computed using
Hasse's famous formula (~\cite{Hasse}):
$$ t = d+q^{\ast}-(r+1+o)$$
where\\
$d$= number of ramified primes in $K/F$, \\
$r$ = {\rm rank} of the free abelian part of the group of units $E_{F}$ of $F$, \\
$o$ = 1 or 0 according as to whether $F$ contains primitive $\ell$th root of unity or not,\\
$q^{\ast}$ is defined by $[ N_{K/F}(K^{\ast} )\cap E_{F} : N_{K/F}(E_{F})] = \ell^{q^{\ast}}$.\\
More generally, let us define for each $i \leq \ell$, $S_K^i$ to be
the subgroup of ambiguous ideal classes in $\lambda^{i-1}S_K$. Thus,
rank $S_K^i = $ rank $\lambda^{i-1}S_K/\lambda^i S_K$ which is the
$\lambda^i$-rank of $S_K$ (which we observed above to be $t-s_1-
\cdots -s_{i-1}$).\\
There is a subtler notion of {\it strongly ambiguous} ideals. An
ambiguous ideal $I$ is said to be strongly ambiguous if the
principal ideal $(1- \tau)I$ is actually $(1)$. This is a notion
for ideal classes. Thus, an ideal class $a\in C_K$ is said to be a
{\it strongly ambiguous ideal class} if there exist a representative
$\mathfrak{a} \in I_K$ for $a$ such that
$(1-\tau) \mathfrak{a} = (1)$. \\
The subgroup $S_{K,s}$ of strongly ambiguous ideal classes in $S_K$
has rank given by a similar formula as above:
$$ {\rm rank} \text{ } S_{K,s} = d + q -(r+1+o) $$
where $q$ is given by $[N_{K/F}(E_{K}) \cap E_{F}: N_{K/F}(E_{F})]=
\ell^{q}$. \vskip 5mm

\section{Cyclic extensions of degree $\ell$}

\noindent In this section we give a formula for the {\rm rank} of $S_{K}$,
where $K/F$ is a Galois extension of odd prime degree $\ell$. \\
{\it Throughout, we assume $S_{F} = \{1\}$.}\\
The proof is an easy generalization of Theorems 3.1 and 4.1 of
\cite{Gerth}.

\begin{prop}\label{prop1}
Let $K$ be a cyclic extension of degree $\ell$ of a number field $F$
for which the $l$-class group $S_{F} = \{1\}$. Let us denote by $t$ the {\rm rank} of the group of
ambiguous ideal classes $S_{K}[\lambda]$ in $S_K$, and by $s_{i}$,
the {\rm rank} of $(\lambda^{i-1}S_{K}[\lambda]+\lambda^{i}S_{K})/\lambda^{i}S_{}$. Then
$${\rm rank} \text{ } S_{K}=
(\ell -1)t - (\ell -2)s_{1}- (\ell -3)s_{2}- \cdots -s_{\ell -2}.$$
Further, $S_{K}$ is isomorphic to the direct product of an
elementary abelian $\ell$-groups of {\rm rank} $s_1$ and an abelian $\ell$ group of {\rm rank}
$$(\ell
-1)(t-s_{1}) - (\ell -3)s_{2}- (\ell -4)s_{3}- \cdots -s_{\ell
-2}.$$
\end{prop}

\begin{proof}
For $R = \mathbb{Z}_l[\zeta]$, the $R$-module decomposition
$$S_K = \oplus_{i=1}^t R/\lambda^{e_i} R = \oplus_i (R/\lambda^i
R)^{s_i}$$ where $1 \leq e_1 \leq \cdots \leq e_t$ and
$$s_i =|\{j: e_j = i \}|,$$
it follows that
$${\rm rank} S_K = \sum_i |\{j: e_j \geq i\}| = t+(t-s_1)+(t-s_1-s_2)+
\cdots + (t-s_1-s_2- \cdots- s_{l-2})$$
$$ = (l-1)t-(l-2)s_1-(l-3)s_2 \cdots - s_{l-2}.$$
In order to get the direct sum decomposition, we consider the
filtration
$$S_K \supset \lambda S_K \supset \cdots \supset \lambda^{l-1}S_K =
lS_K$$ and the homomorphism
$$\lambda_{i}^{\ast} : \lambda^{i-1}S_{K}/\lambda^{i}S_{K} \twoheadrightarrow \lambda^{i}S_{K}/\lambda^{i+1}S_K$$
induced by multiplication by $\lambda$.\\
As
$\lambda^{i-1}S_{K}/\lambda^{i}S_{K}$ are elementary abelian
$\ell$-groups, they can be viewed as vector spaces over
$\mathbb{F}_{l}$. Then $\lambda_{i}^{\ast}$ is a surjective, vector
space homomorphism. Hence there exists groups $R_{i},W_{i}$ such
that
$$\lambda^{i} S_{K} \subset R_{i},W_{i} \subset \lambda^{i-1} S_{K}$$
and so that $\lambda_i^{\ast}$ gives an isomorphism between
$R_{i}/\lambda^{i}S_{K}$ and $\lambda^{i}S_{K}/\lambda^{i+1}S_{K}$
and
$$ ker \lambda_{i}^{\ast} = W_{i}/\lambda^{i} S_{K}.$$
Therefore,
$$ R_{i}+W_{i} = \lambda^{i-1}S_{K} \text{     and     } R_{i} \cap W_{i} = \lambda^{i}
S_{K}.$$ Clearly $W_i = (\lambda^{i-1}S_{K})[\lambda] +
\lambda^{i}S_{K}$ from the definition of $\lambda_i^{\ast}$.\\
So, there exists a subgroup $H_{i} \subset
(\lambda^{i-1}S_{K})[\lambda]$, such that
$$ W_{i}= H_{i} \oplus \lambda^{i} S_{K},$$
with $H_{i} \cong (\lambda^{i-1}S_{K}[\lambda] + \lambda^{i}S_{K})/\lambda^{i}S_{K}$. Then
$$\lambda^{i-1}S_{K} = R_{i}+W_{i} = R_{i} + (H_{i} \oplus \lambda^{i}S_{K}) \cong R_{i} \oplus H_{i}$$
since  $R_{i} \cap W_{i} = \lambda^{i}S_{K}$. In particular, for
$i=1$, we get,
$$ S_{K} \cong R_{1} \oplus H_{1}.$$

\noindent Recall that $s_{i} = {\rm rank}
(\lambda^{i-1}S_{K}[\lambda] + \lambda^{i}S_{K})/\lambda^{i}S_{K}$;
thus
$$s_i = {\rm rank} \text{ } H_{i}  =
{\rm rank} \text{ } W_{i}/\lambda^{i}S_{K}.$$ {\it Thus, the
proposition will follow if we can prove:} $${\rm rank}~ R_{1} =
(\ell -1)(t-s_{1}) - (\ell -3)s_{2}- (\ell -4)s_{3}- \cdots -
s_{\ell -2}.$$ Since $\ell H_{i} = \{ 1\}$ and $\ell S_{K}
=\lambda^{\ell-1} S_{K}$,
we have:\\
$${\rm rank} R_{1} = {\rm rank} R_{1}/\ell R_{1} = {\rm rank} R_{1}/\ell S_{K} = {\rm rank} R_{1}/ \lambda^{\ell-1}S_{K}$$
$$ = {\rm rank} R_{1}/\lambda S_{K} + {\rm rank} \lambda S_{K}/\lambda^{2}S_{K} +\cdots + {\rm rank} \lambda^{\ell-2}S_{K}/\lambda^{\ell-1}S_{K}$$
$$= 2\cdot {\rm rank} R_{1}/\lambda S_{K} + {\rm rank} R_{2}/\lambda^{2}S_{K} + \cdots + {\rm rank} R_{\ell-2}/\lambda^{\ell-2}S_{K}, $$
since  $R_{i}/\lambda^{i}S_{K} \cong \lambda^{i}S_{K}/\lambda^{i+1}S_{K}$.\\
Now,
 $${\rm rank} R_{1}/\lambda S_{K} = {\rm rank} S_{K}/\lambda S_{K} - {\rm rank} W_{1}/\lambda S_{K} =
 t-s_{1}.$$
Similarly,
$${\rm rank} R_{i}/\lambda^{i}S_{K} = {\rm rank} \lambda^{i-1}S_{K}/\lambda^{i}S_{K} -
{\rm rank} W_{i}/\lambda^{i}S_{K}$$
 $$ = {\rm rank}
R_{i-1}/\lambda^{i-1}S_{K} -s_{i} = t-s_{1}-s_{2} -\cdots
-s_{i}.$$ Putting all of these together, we get
$${\rm rank} R_{1} = (\ell -1)(t-s_{1}) - (\ell -3)s_{2}- (\ell -4)s_{3}- \cdots -s_{\ell -2}$$ and
$${\rm rank} S_{K} = (\ell -1)t - (\ell -2)s_{1}- (\ell -3)s_{2}- \cdots -s_{\ell -2}.$$
\end{proof}

\noindent {\bf Remark.} We saw in the beginning of the proof above
that from the decomposition $S_K \cong R/\lambda^{e_{1}}R \times
\cdots \times R/\lambda^{e_{t}}R$, where $R = \mathbb{Z}_l[\zeta_l]$
and $\lambda =  1- \zeta_l$, one can easily find the formula of the
rank by simple counting. We have given the above proof as some
ingredients of the proof like the subgroups $R_i$ and $W_i$ will be
used later in the construction of genus fields.

\noindent We point out the following special cases of interest; the
first corollary below is immediate:

\begin{cor}
If $t=s_{1}$, then $S_{K}$ is an elementary abelian $\ell$ group of
{\rm rank} $t$.
\end{cor}

\begin{cor}
For $i \geq 1$, we have
$$ {\rm rank} \lambda^i S_{K}/\lambda^{i+1}S_{K} = t - s_{1}-\cdots -s_{i}.$$
In particular, $t-s_{1}-\cdots -s_{i} \geq 0$ for all $i$ and so, we
observe that,
$$0 \leq s_{i} \leq t-s_{1} - \cdots -s_{i-1}.$$
\end{cor}

\begin{proof}
The proof of this corollary is contained in the proof of Theorem
\ref{prop1}.
\end{proof}

\begin{cor}
For some $1 \leq i \leq (\ell-2)$, if we have $\lambda^i S_{K} =
lS_{K}$, then $s_{j} = 0$ for $j \geq i+1$ and $t=s_{1}+\cdots
+s_{i}$.
\end{cor}

\begin{proof}
Since $\lambda^i S_{K}^{\Delta^{i}} = l S_{K}$, we see that
$$\lambda^i S_{K} = \lambda^{i+1} S_{K} = \cdots = \lambda^{l-1}S_{K} = l S_{K}.$$
So, the quotients $\lambda^j S_{K}/\lambda^{j+1}S_{K}$ are trivial
for all $j \geq i$. The previous corollary implies the assertion
now.
\end{proof}

\begin{cor}
The {\rm rank} of $S_{K}$ satisfies the bounds
$$ 2t-s_{1} \leq {\rm rank} S_{K} \leq (\ell-1)t - (\ell-2)s_{1}.$$
Moreover, if $s_{2}=t-s_{1}$, then the lower bound is achieved; that
is, {\rm rank} of $S_{K}$ equals $2t-s_{1}$. Further, if
$s_{2}=s_{3}=\cdots=s_{\ell-2}=0$, then the upper bound is achieved,
that is,{\rm rank} of $S_{K}$ is $(\ell-1)t - (\ell-2)s_{1}$.
\end{cor}

\begin{proof}
The upper bound of {\rm rank} $S_{K}$ is immediate from the proposition because
$ {\rm rank} S_{K}=(l-1)t-(l-2)s_1- \cdots - s_{l-2}$. The lower bound  follows since
${\rm rank} S_{K} = t+ \sum_{i=1}^{\ell-2}(t-s_{1}-\cdots-s_{i}) \geq 2t-s_{1}$ (since $t-s_{1}-\cdots -s_{i} \geq 0$). Combining these two facts we obtain the bound for {\rm rank} of $S_{K}$.\\
Since $t-s_{1}-s_{2}=0$, we see that
$s_{3}=s_{4}=\cdots=s_{\ell-2}=0$ (follows from Corollary 3.3).
Substituting these values of $s_{i}$ in the formula for the {\rm
rank} of $S_{K}$ in Theorem \ref{prop1}, we obtain that, ${\rm rank}
S_{K}= 2t-s_{1}$.
\end{proof}

\noindent {\bf Remark.} The above bounds constitute an improvement
of the bounds obtained in Gerth's paper \cite{Gerth}[Corollary 2.5];
he obtains, $t \leq {\rm rank} S_{K} \leq (\ell-1)t$. \vskip 5mm

\section{When $F$ contains $\zeta_{\ell}$ and has class number
coprime to $\ell$}

\noindent
We recall the earlier notations: \\
$K$ is a cyclic extension of degree $l$ (an odd prime) over a number
field $F$ with trivial $l$-class group which, we now assume,
contains a primitive $l$-th roots of unity $\zeta$. By class field
theory, the maximal abelian unramified extension $M_0$ of $K$
satisfies $C_K \cong {\rm Gal}(M_0/K)$. The genus field of $K/F$ is
the maximal abelian extension $M$ of $F$ which is contained in
$M_0$; then ${\rm Gal}(M/F)$ is the abelianization of ${\rm
Gal}(M_0/F)$. Moreover, $C_K/\lambda C_K \cong {\rm Gal} (M/K)$ and
is called the group of genera. By Kummer theory, $K = F(x^{1/l})$
for some $x \in F^{\ast} - (F^{\ast})^l$. Being the $l$-Sylow
subgroup of the group $C_K/\lambda C_K$, the group $S_K/\lambda S_K$
is a direct summand of it. Thus, there is a unique subfield $M_1$ of
$M$ which contains $K$ and satisfies $S_K/\lambda S_K \cong {\rm
Gal}(M_1/K)$; thus,
note that Gal $(M_1/K)$ is elementary abelian of rank $t$.\\
Recall also from the proof of \ref{prop1} that for $i \geq 1$, there
is a subgroup $H_i \subset (\lambda^{i-1}S_K)[\lambda]$ such that
$H_i \cap \lambda^i S_K = (0)$ and $H_i \cong
((\lambda^{i-1}S_K)[\lambda] + \lambda^i S_K)/\lambda^i S_K$. Note
that $s_i = {\rm rank} H_i$ for $i \geq 1$.\\
The first theorem below computes the rank $s_1$ of $H_1$ in terms of
the rank of a certain matrix with entries in $\mathbb{F}_l$.\\
Firstly, by Kummer theory, there exist $x_{1},\cdots,x_{t} \in
K^{\ast} - (K^{\ast})^l$ such that $M_{1}=
K(x_{1}^{\frac{1}{\ell}},\cdots,x_{t}^{\frac{1}{\ell}})$. In the
following theorem, we obtain a $t \times t$ matrix over
$\mathbb{F}_{\ell}$ whose {\rm rank} equals $s_1$. The entries of
this matrix involve the Artin symbols of the generators $x_i$'s.\\
Note that the Artin symbols
$\bigg(\frac{K(x_{i}^{\frac{1}{\ell}})/K} {I}\bigg)$ are defined for
any ideal $I$ of $K$ since the conductor of the field $M_1$ is
trivial.

\begin{thm}\label{thm2}
Let $F,K,M,M_{1}$ be as above. Fix representative ideals
$\mathfrak{a}_{1},\cdots,\mathfrak{a}_{t}$ whose ideal classes form
a basis for the group $S_{K}[\lambda]$. Denote by $\mu_{ij}$, the
Artin symbol $\bigg(\frac{K(x_{i}^{\frac{1}{\ell}})/K}
{\mathfrak{a}_{j}}\bigg).$ Write $\alpha_{ij} \in \mathbb{F}_l$ for
which $\zeta^{\alpha_{ij}}$ is the power residue symbol
$(x_{i}^{\frac{1}{\ell}})^{\mu_{ij}-1}$. If $A_1$ is the matrix
$(\alpha_{ij}) \in M_t(\mathbb{F}_l)$, then $${\rm rank} A_{1}= {\rm
rank} H_1 = s_1.$$
\end{thm}

\begin{proof}
As noted above, since $M$ is an unramified extension of $K$ and $K
\subset M_{1} \subset M$, the conductor of $M_{1}/K$ is trivial and
hence, the Artin symbol $\big( \frac{M_{1}/K}{\mathfrak{a}} \big)
\in {\rm Gal}(M_{1}/K)$ is well defined for all ideals
$\mathfrak{a}$ of $K$. We define a map

$$ \psi_{1}: S_{K}[\lambda] \to S_{K} \to S_{K}/\lambda S_{K} \cong
{\rm Gal}(M_{1}/K)$$ which is the composite of the natural
inclusion, the natural surjection and the canonical isomorphism. If
$cl(\mathfrak{a})$ denotes the ideal class of an ideal
$\mathfrak{a}$, and if $cl(\mathfrak{a}) \in S_{K}[\lambda]$, then
we see by Artin reciprocity that $\psi_{1}(cl(\mathfrak{a}))= \big(
\frac{M_{1}/K}{\mathfrak{a}} \big)$ and that
the kernel of $\psi_{1}$ is $S_{K}[\lambda] \cap \lambda S_{K}$.\\
Now $M_{1}= K(x_{1}^{\frac{1}{\ell}},\cdots,x_{t}^{\frac{1}{\ell}})$
and $[M_{1}:K] =\ell^{t}$ imply that there exists an isomorphism
$$\delta_{1}: {\rm Gal}(M_{1}/K)\cong
{\rm Gal}(K(x_{1}^{\frac{1}{\ell}})/K) \times \cdots \times {\rm Gal}(K(x_{t}^{\frac{1}{\ell}})/K).$$
For each $i=1,\cdots,t$, Kummer theory provides an isomorphism
$$\theta_{i}: {\rm Gal}(K(x_{i}^{\frac{1}{\ell}})/K) \to \mathbb{F}_{\ell}$$
$$ \mu \mapsto \alpha_{\mu}$$
where $\zeta^{\alpha_{\mu}} = (x_{i}^{\frac{1}{\ell}})^{\mu-1}$.\\
Define
$$ \phi_{1} := \big( \prod_{i=1}^{t} \theta_{i} \big) \circ \delta_{1} \circ
\psi_{1} : S_{K}[\lambda]  \to \mathbb{F}_{\ell}^{t}.$$ Now,
$S_{K}[\lambda]$ is a vector space over $\mathbb{F}_{\ell}$ (as it
is an elementary abelian $\ell$-group) and $\phi_{1}$ is a vector
space homomorphism;
also $ker \phi_{1} = ker{\psi_{1}} = S_{K}[\lambda] \cap \lambda S_{K}$.\\
Now $A_{1}$ is precisely the matrix of $\phi_{1}$ with respect to
basis $\{ cl(\mathfrak{a_{1}}), \cdots, cl(\mathfrak{a_{t}}) \}$ of
$S_{K}[\lambda]$. Then

$$ {\rm rank}(S_{K}[\lambda] \cap \lambda S_{K}) = {\rm rank}(ker(\phi_{1})) = t- {\rm rank} A_{1}.$$
Equivalently, ${\rm rank} A_{1} = t- {\rm rank}(S_{K}[\lambda] \cap
\lambda S_{K}) $. Since $S_{K}[\lambda]$ is an elementary abelian
$\ell$-group of {\rm rank} $t$ and
$$ H_1 \cong (S_{K}[\lambda] + \lambda S_{K})/\lambda S_{K} \cong S_{K}[\lambda]/ (S_{K}[\lambda] \cap \lambda S_{K}),  $$
then $s_1 = {\rm rank} H_1 = t- {\rm rank}(S_{K}[\lambda] \cap
\lambda S_{K}) = {\rm rank} A_{1}$.
\end{proof}
\vskip 5mm

\noindent The above result for the rank $s_1$ of $H_1$ can be
generalized to a general $s_i$ in the following manner.\\
Recall from the proof of proposition 3.1 that there exists a
subgroup $R_i$ satisfying $\lambda^i \subset R_i \subset
\lambda^{i-1}S_K$ and
$$R_i/\lambda^i S_K \cong \lambda^{i-1}S_K/W_i \cong \lambda^i
S_K/\lambda^{i+1}S_K$$ where the last isomorphism is induced by the
multiplication-by-$\lambda$ map. \\
Now, for $1 \leq i \leq \ell-3$, if we have chosen a genus field
$M_i \subset M$ with ${\rm Gal}(M_{i}/K) \cong
\lambda^{i-1}S_{K}/\lambda^i S_{K}$, there exists - corresponding to
the direct summand $R_{i}/\lambda^i S_{K}^{\Delta^{i}}$ - a unique
field $M_{i+1}$ such that
$$K \subset M_{i+1} \subset M_{i} \subset
M_{1} \subset M~~ and$$ $${\rm Gal}(M_{i+1}/K) \cong R_{i}/\lambda^i
S_{K} \cong \lambda^i S_{K}/\lambda^{i+1}S_{K}.$$ From the above
isomorphism, we see that ${\rm Gal}(M_{i+1}/K)$ is an elementary
abelian $\ell$-group. We have
$$t-s_{1}-\cdots -s_{i}= {\rm rank} ~{\rm Gal}(M_{i+1}/K) = {\rm rank} \lambda^i
S_{K}/\lambda^{i+1}S_{K}.$$ Once again, Kummer theory assures us
elements $y_{1},\cdots,y_{t-s_{1}-\cdots-s_{i}} \in K^{\ast} -
(K^{\ast})^l$ such that $$M_{i+1}=
K(y_{1}^{\frac{1}{\ell}},\cdots,y_{t-s_{1}-\cdots-s_{i}}^{\frac{1}{\ell}}).$$
Fix representative ideals
$\mathfrak{b}_{1},\cdots,\mathfrak{b}_{t-s_{1}-\cdots-s_{i}}$ whose
ideal classes form a basis for the group $(\lambda^i
S_{K})[\lambda]$. With these notations, we prove the following
theorem: \vskip 5mm

\begin{thm}\label{thm3}
Let $\mu_{jk}$ denote the Artin symbol
$\bigg(\frac{K(y_{j}^{\frac{1}{\ell}})/K}
{\mathfrak{b}_{k}}\bigg)$, and let $\beta_{jk} \in \mathbb{F}_l$ for
which $\zeta^{\beta_{jk}}$ is the power residue symbol
$(y_{j}^{\frac{1}{\ell}})^{\mu_{jk}-1}$. For $1 \leq i \leq l-3$, if
$A_{i+1}$ is the matrix $(\beta_{jk}), 1 \leq j,k \leq t-s_1-
\cdots-s_i$ with entries in $\mathbb{F}_l$, then
$$s_{i+1} = {\rm rank} A_{i+1}.$$
\end{thm}

\begin{proof}
Since $M$ is an unramified extension of $K$ and $K \subset M_{i+1} \subset M$,
the conductor of $M_{i+1}/K$ is $(1)$ and hence the Artin symbol
$\big( \frac{M_{i+1}/K}{\mathfrak{a}} \big) \in {\rm Gal}(M_{i+1}/K)$
is well defined for all ideals $\mathfrak{a}$ of $K$. We define a map

$$ \psi_{i+1}: (\lambda^i S_{K})[\lambda] \to \lambda^i S_{K} \to \lambda^i S_{K}/\lambda^{i+1}S_{K}
\cong R_{i}/\lambda^i S_{K} \cong {\rm Gal}(M_{i+1}/K)$$ which is
the composite of the natural inclusion, the natural surjection and
the canonical isomorphisms. If $cl(\mathfrak{a}) \in (\lambda^i
S_{K})[\lambda]$, then by Artin reciprocity, we see that
$\psi_{i+1}(cl(\mathfrak{a}))= \big( \frac{M_{i+1}/K}{\mathfrak{a}}
\big)$
and that the kernel of $\psi_{i+1}$ is $(\lambda^i S_{K})[\lambda] \cap \lambda^{i+1}S_{K}$.\\
Now $M_{i+1}= K(y_{1}^{\frac{1}{\ell}},\cdots,y_{t-s_{1}-\cdots -s_{i}}^{\frac{1}{\ell}})$ and $[M_{i+1}:K] =\ell^{t-s_{1}-\cdots -s_{i}}$ imply that there exists an isomorphism
$$\delta_{i+1}: {\rm Gal}(M_{i+1}/K)\cong {\rm Gal}(K(y_{1}^{\frac{1}{\ell}})/K) \times \cdots \times {\rm Gal}(K(y_{t-s_{1}-\cdots-s_{i}}^{\frac{1}{\ell}})/K).$$
Once again, Kummer theory provides for each $j \leq
t-s_{1}-\cdots-s_{i}$, an isomorphism
$$\theta_{j}: {\rm Gal}(K(y_{j}^{\frac{1}{\ell}})/K) \to \mathbb{F}_{\ell}$$
$$ \mu \mapsto \alpha_{\mu}$$
where $\zeta^{\alpha_{\mu}} = (y_{j}^{\frac{1}{\ell}})^{\mu-1}$. \\
Define
$$ \phi_{i+1} := \big( \prod_{j=1}^{t-s_{1}-\cdots -s_{i}} \theta_{j} \big) \circ \delta_{i+1} \circ \psi_{i+1} :
(\lambda^i S_{K})[\lambda]  \to \mathbb{F}_{\ell}^{t-s_{1}-\cdots
-s_{i}}.$$ Since $(\lambda^i S_{K})[\lambda]$ is an elementary
abelian $\ell$-group, it may be viewed as a vector space over
$\mathbb{F}_{\ell}$ and, $\phi_{i+1}$ is a vector space
homomorphism. Since $$ker \phi_{i+1} = ker{\psi_{i+1}} = (\lambda^i
S_{K})[\lambda] \cap \lambda^{i+1}S_{K}$$ and since $A_{i+1}$ is
precisely the matrix of $\phi_{i+1}$ with respect to the basis $\{
cl(\mathfrak{b_{1}}), \cdots, cl(\mathfrak{b_{t-s_{1}-\cdots
-s_{i}}}) \}$, we have
$${\rm rank} ((\lambda^i S_{K})[\lambda] \cap \lambda^{i+1}S_{K})
= {\rm rank}(ker(\phi_{i+1})) = (t-s_{1}-\cdots -s_{i})- {\rm rank}
A_{i+1}$$ Equivalently, ${\rm rank} A_{i+1} =(t-s_{1}-\cdots-s_{i})-
{\rm rank} ((\lambda^i S_{K})[\lambda] \cap \lambda^{i+1}S_{K})$.\\
Since $(\lambda^i S_{K})[\lambda]$ is an elementary abelian
$\ell$-group of {\rm rank} $(t-s_{1}-\cdots -s_{i})$ and
$$ ((\lambda^i S_{K})[\lambda] + \lambda^{i+1}S_{K})/\lambda^{i+1}S_{K} \cong
(\lambda^i S_K)[\lambda]/(\lambda^i S_{K})[\lambda] \cap
\lambda^{i+1}S_{K} \cong H_{i+1},$$ we obtain
$$s_{i+1} = {\rm rank} H_{i+1} = (t-s_{1}-\cdots
-s_{i}) - {\rm rank}((\lambda^i S_{K})[\lambda] \cap
\lambda^{i+1}S_{K}) = {\rm rank} A_{i+1}.$$
\end{proof}
\vskip 5mm

\section{$F = \Q(\zeta_{5})$ and $K = F(x^{1/5})$}

\noindent In the last section, we showed that the {\rm rank} of
$S_K$ can be expressed in terms of the ranks of certain matrices
over $\mathbb{F}_{\ell}$. The explicit determination of these
matrices seems very difficult in general. Gerth had carried this out
in the case of $\ell=3$. In this section, we look at the case of
$\ell=5$. We assume in this section that $F$ is the cyclotomic field
$\Q(\zeta)$ generated by the $5$-th roots of unity; note that $F$
has class number $1$. We analyze what the earlier theorems give for
several examples and compare them with computations obtained by the
program SAGE (the latter uses GRH) - a detailed table is given at
the end of the paper. After that, we exploit the theorems of the
previous section to prove a number of general results which were
guessed at by a close inspection of the tables. \vskip 5mm

\noindent Consider any cyclic extension $K=F(x^{\frac{1}{5}})$ of
degree $5$ over $F$. We may assume that $x$ is an integer in $F$
which is not
divisible by the $5^{th}$ power of any prime element of $F$.\\
The ring of integer $\Z[\zeta]$ of $F$ is a principal ideal domain.
Consider those nonzero elements $x$ which can be written as
$$x = u \lambda^{e_{\lambda}} \pi_{1}^{e_{1}} \cdots \pi_{g}^{e_{g}},$$
where $u$ is a unit in $\Z[\zeta]$, $\lambda = 1-\zeta$ is the
unique prime over $5$ (so, $\lambda^{4} || 5$), and
$\pi_{1},\cdots,\pi_{g}$ are prime elements in $F$, where $e_i \in
\{1,\cdots,4\}$ for $1\leq i \leq g$, and $e_{\lambda} \in
\{0,1,\cdots,4\}$. \vskip 3mm

\subsection{Unwinding Hasse's formula for
$\mathbb{Q}(\zeta,x^{1/5})$} \vskip 3mm

\noindent Let us see how the {\rm rank} $t$ of the group of
ambiguous ideal classes in the $5$-class group $S_{K}$ is computed
using Hasse's famous formula (~\cite{Hasse}) in our case:
$$ t = d+q^{\ast}-(r+1+o).$$
For our fields $F$ and $K$, we have
$$ r=\frac{\ell-3}{2}=1,$$
$$o=1,$$
\begin{equation*}
d =
\begin{cases}
g & \text{ if } (\lambda)\text{ does not ramify in }K/F,\\
g+1 & \text{ if } (\lambda)\text{ ramifies in }K/F,
\end{cases}
\end{equation*}

\begin{equation*}
q^{\ast}   \leq \frac{\ell-1}{2}=2, \text{ (since the order of
}[E_{F}:E_{F}^{\ell}] = \ell^{\frac{\ell-1}{2}}=5^2)
\end{equation*}
\vskip 5mm

\noindent In $F$, the group of units is generated by $\zeta$ and
$1+\zeta$ where $\zeta$ is a primitive $5$-th root of unity. We see
from the definition of $q^{\ast}$ that
\begin{equation*}
q^{\ast} =
\begin{cases}
2 \text{ if } \zeta, 1+\zeta \in N_{K/F}(K^{\ast}),\\
1  \text{ if some, but not all } \zeta^{i}(1+\zeta)^{j} \in N_{K/F}(K^{\ast}),\\
0 \text{ if } \zeta^{i}(1+\zeta)^{j} \notin N_{K/F}(K^{\ast}), \text{ for } 0 \leq i,j \leq 4, i+j \neq 0.
\end{cases}
\end{equation*}

\noindent We have $t= d-3+q^{\ast}$ with $d$, $q^{\ast}$ determined
as above. \vskip 5mm

\noindent Since $q^{\ast}$ depends on whether
$\zeta^{i}(1+\zeta)^{j}$ is a norm from $K$ or not, its value can be
determined in terms of the local Hilbert symbols in completions of
$F$ as in the following lemma. \vskip 3mm

\begin{lem}
Let $F = \mathbb{Q}(\zeta)$ and let $K=F(x^{1/5})$ where $x = u
\lambda^{e_{\lambda}} \pi_{1}^{e_{1}} \cdots \pi_{g}^{e_{g}},$ with
$u$ a unit in $\Z[\zeta]$, $\lambda = 1-\zeta$ is the unique prime
over $\ell$ (so, $\lambda^{4} || 5$), and $\pi_{1},\cdots,\pi_{g}$
prime elements in $\Z[\zeta]$. Then\\
(a) $\zeta \in N_{K/F}(K^{\ast}) \iff N_{F/\Q}((\pi_{i})) \equiv
1\pmod{25}$ for all $i$;\\
(b) if $\zeta^{i}(1+\zeta)^{j} \in N_{K/F}(K^{\ast})$, if and only
if, every $\pi_k|x$ above has the property that $\zeta^i(1+\zeta)^j$
is a $5$-th power modulo $(\pi_k)$ in $\mathbb{Z}[\zeta]$ for all
$i,j$;\\
(c) $(\lambda)$ ramifies in $K/F \iff x \not\equiv \pm 1, \pm
7\pmod{ \lambda^{5}}$.
\end{lem}

\begin{proof}
Now $\zeta^{i}(1+\zeta)^{j} \in N_{K/F}(K^{\ast}) \iff \big(
\frac{x,\zeta^{i}(1+\zeta)^{j}}{\mathfrak{p}} \big) =1$ for all
prime ideals
$\mathfrak{p}$ of $F$. \\
Since $\zeta^{i}(1+\zeta)^{j}$ is a unit, $\big( \frac{x,\zeta^{i}(1+\zeta)^{j}}{\mathfrak{p}} \big)
=1$ if $\mathfrak{p}$ does not ramify in $K/F$. We will now look at the prime ideals $(\lambda)$
and the $(\pi_k)$'s.\\
Firstly, look at $\mathfrak{p}=(\pi_{k}),$ where $\pi_{k}|x$. Then
$$\bigg( \frac{x,\zeta^{i}(1+\zeta)^{j}}{(\pi_{k})} \bigg) =
\bigg( \frac{u \lambda^{e_{\lambda}} \pi_{1}^{e_{1}} \cdots
\pi_{g}^{e_{g}},\zeta^{i}(1+\zeta)^{j}}{(\pi_{k})} \bigg) $$
$$=  \bigg( \frac{u \lambda^{e_{\lambda}},\zeta^{i}(1+\zeta)^{j}}{(\pi_{k})} \bigg)
\bigg( \frac{ \pi_{1}^{e_{1}},\zeta^{i}(1+\zeta)^{j}}{(\pi_{k})} \bigg) \cdots
\bigg( \frac{ \pi_{g}^{e_{g}},\zeta^{i}(1+\zeta)^{j}}{(\pi_{k})} \bigg) $$
$$= \bigg( \frac{ \pi_{k},\zeta^{i}(1+\zeta)^{j}}{(\pi_{k})} \bigg)^{e_{k}}=1$$
$$\iff \bigg( \frac{\pi_{k},\zeta^{i}(1+\zeta)^{j}}{(\pi_{k})}
\bigg) =1 \iff \bigg(
\frac{\zeta^{i}(1+\zeta)^{j},\pi_{k}}{(\pi_{k})} \bigg) =1.$$ The
last equality is equivalent to the conditions
$$ \pi_{k} \text{ splits completely in } F((\zeta^{i}(1+\zeta)^{j})^{\frac{1}{5}})/F.$$
Since the last condition holds if and only if
$\zeta^{i}(1+\zeta)^{j} \equiv a^{5}\pmod{ (\pi_{k})} \text{ for some
} a \in \Z[\zeta]$ (\cite{Hecke}[Theorem 118]), the necessity
assertion in (b) for $\zeta^i(1+\zeta)^j$ to be a norm follows.\\
The converse assertion follows by the product law since $ \bigg(
\frac{x,\zeta^{i}(1+\zeta)^{j}}{(\pi_{k})} \bigg)=1$ for all
$\pi_{k}|x$ implies $ \bigg(
\frac{x,\zeta^{i}(1+\zeta)^{j}}{(\lambda)} \bigg)=1$ and
hence, $\zeta^{i}(1+\zeta)^{j} \in N_{K/F}(K^{\ast}) $.\\
To deduce (a) from (b), we note that, in particular, $\zeta \in
N_{K/F}(K^{\ast})$ if and only if, all $\pi_{i}$ splits completely
in $F(\zeta^{\frac{1}{5}})/F$,
which is equivalent to the condition, $N_{F/\Q}((\pi_{i})) \equiv 1\pmod{25}$, for all $i$. \\
Finally (c) follows from (\cite{Hecke}[Theorem 119]) which shows
that $(\lambda)$ ramifies in $K/F$ iff $x \not\equiv \pm 1, \pm
7\pmod{ \lambda^{5}}$.
\end{proof}
\vskip 5mm

\noindent From the above lemma, we may immediately formulate the
following proposition. \vskip 3mm

\begin{prop}
Let $F=\mathbb{Q}(\zeta)$ and $K=F(x^{\frac{1}{5}})$ of degree $5$
as above. Write $x= u \lambda^{e_{\lambda}} \pi_{1}^{e_{1}} \cdots
\pi_{g}^{e_{g}}$, $\lambda = 1-\zeta$, each $\pi_{i}\in F$ is a
prime element, $e_{i} \in \{1,2,3,4\}$ for $1\leq i \leq g$,
$e_{\lambda} \in \{0,1,2,3,4\}$ as before. Let $d$ denote the number
of primes that ramify in $K/F$. Then the {\rm rank} $t$ of the group
of ambiguous ideal classes in $S_{K}$ is given by:\\
$t = d-1, d-2$ or $d-3$ respectively as to the following three
situations I,II,III:\\
I : each $\pi_k|x$ has the property that both $\zeta \text{ and }
1+\zeta$ are $5$-th powers modulo $(\pi_{k})$ in
$\mathbb{Z}[\zeta]$; \\
II : each $\pi_k|x$ has the property that some, but not all
$\zeta^{i}(1+\zeta)^{j}$ is a $5$-th power modulo $(\pi_{k})$ in
$\mathbb{Z}[\zeta]$; \\
III: some $\pi_k|x$ has the property that none of the
$\zeta^{i}(1+\zeta)^{j}$ is a $5$-th power modulo $(\pi_{k})$ in
$\mathbb{Z}[\zeta]$.\\
These are further simplified to the expressions $t = g,g-1,g-2$ or
$g-3$ according as to the respective conditions A,B,C,D:\\
A : \\
$x \not\equiv\pm 1, \pm 7$ modulo  $(\lambda^{5})$ and each
$\pi_k|x$ has the property that both $\zeta$ and $1+\zeta$ are
$5$-th powers modulo $(\pi_{k}))$ in $\mathbb{Z}[\zeta]$; \\
B : \\
$x \equiv\pm 1, \pm 7$ modulo  $(\lambda^{5})$ and each $\pi_k|x$
has the property that both $\zeta$ and $1+\zeta$ are $5$-th powers
modulo $(\pi_{k}))$ in $\mathbb{Z}[\zeta]$;\\
OR\\
$x \not\equiv\pm 1, \pm 7$ modulo  $(\lambda^{5})$ and each
$\pi_k|x$ has the property that some, but not all
$\zeta^{i}(1+\zeta)^{j}$ is a $5$-th power modulo $(\pi_{k})$ in
$\mathbb{Z}[\zeta]$;\\
C : \\
$x \not\equiv\pm 1, \pm 7$ modulo  $(\lambda^{5})$ and some
$\pi_k|x$ has the property that none of the $\zeta^{i}(1+\zeta)^{j}$
is a $5$-th power modulo $(\pi_{k})$ in $\mathbb{Z}[\zeta]$;\\
OR\\
$x \equiv\pm 1, \pm 7$ modulo  $(\lambda^{5})$ and each $\pi_k|x$
has the property that some, but not all $\zeta^{i}(1+\zeta)^{j}$ is
a $5$-th power modulo $(\pi_{k})$ in $\mathbb{Z}[\zeta]$;\\
D :\\
$x \equiv\pm 1, \pm 7$ modulo  $(\lambda^{5})$ and some $\pi_k|x$
has the property that none of the $\zeta^{i}(1+\zeta)^{j}$ is a
$5$-th power modulo $(\pi_{k})$ in $\mathbb{Z}[\zeta]$.
\end{prop}
\vskip 5mm

\subsection{Examples}
\vskip 3mm

\begin{example}
We consider $K= \Q(\zeta,7^{\frac{1}{5}})$.  We observe that $\zeta
\in N_{K/F}(K^{\ast})$, and $(1+\zeta) \equiv (2+4 \zeta^{3})^{5}
\pmod{7}$. Hence $q^{\ast} =2$, and $t=g-1=0$. Thus, $S_{K}$ must be
trivial by our result. This is confirmed by a SAGE program (which is
known to be valid under GRH) - see table 1 compiled in section 5.5.
\end{example}
\vskip 5mm

\begin{example}
Let $K= \Q(\zeta,18^{\frac{1}{5}})$. Notice that in this case $\zeta
\notin N_{K/F}(K^{\ast})$. We observe that,
$$ \zeta^{2}(1+\zeta) \equiv (1+\zeta^{2})^{5} \pmod{2},$$
$$ \zeta^{2}(1+\zeta) \equiv (-1-\zeta)^{5} \pmod{3}.$$
Hence $q^{\ast}=1$ and $t=g-2=0$. Thus, $S_{K}=\{1\}$ from our
result. This is again confirmed by a SAGE program as seen from the
table in 5.5.
\end{example}
\vskip 5mm

\begin{example}
Let $K= \Q(\zeta,11^{\frac{1}{5}})$. Note that in $F=\Q(\zeta_{5})$,
$11= \pi_{1}\pi_{2}\pi_{3}\pi_{4}$, with
$\pi_{1}=(\zeta^3+2\zeta^2+\zeta+2) = (\zeta+2)$,
$\pi_{2}=(-\zeta^2+\zeta+1)$, $\pi_{3}=(\zeta^3-\zeta+1)$,
$\pi_{4}=(-2\zeta^3-\zeta^2-\zeta)=(2\zeta^2+\zeta+1)$. Thus in this
case, $g=4$.\\
Next, we note that $\zeta \equiv -2\pmod {(\zeta+2)}$ is not a
$5^{th}$ power in $\Z[\zeta]$, because if it is a $5^{th}$ power
modulo $(\zeta+2)$, then $11| (n^{5}+2)$, for some integer $n \in
\Z$,
which is not true.\\
We also notice that $\zeta \not\equiv x^{5} \pmod{
(-\zeta^2+\zeta+1)}$, because if so, then we have modulo
$-\zeta^2+\zeta+1$,
$$ \zeta \equiv (a+b\zeta)^{5} \equiv (a^5+b^5+10a^3b^2+10a^2b^3+10ab^4)+
5ab(a^3+2a^2b+4ab^2+3b^3)\zeta.$$ Next we note that, for $1\leq
i\leq 4$, we have,
$$ (1+\zeta)^{i} \equiv \zeta^{2i} \pmod{ (-\zeta^2+\zeta+1)}.$$
Hence $(1+\zeta)^{i}$ is not a $5^{th}$ power modulo $(-\zeta^2+\zeta+1)$.\\
Next for $1 \leq i \leq 4$, $ 0 \leq j \leq 4$, we see that,
$$\zeta^{i}(1+\zeta)^{j} \equiv (-1)^{j}\zeta^{i} \pmod{ (\zeta+2)}.$$
Hence, $\zeta^{i}(1+\zeta)^{j}$ is not a $5^{th}$ power modulo $(\zeta+2)$. \\
So in this case we have $q^{\ast}=0$ and $t=g-2=2$. So {\rm rank}
$S_{K} \geq 2$ by our result. This is confirmed by SAGE - see table
1.
\end{example}
\vskip 5mm

\begin{example}
Let $K= \Q(\zeta,19^{\frac{1}{5}})$. Note that in $F=\Q(\zeta_{5})$,
$19= \pi_{1}\pi_{2}$, with
$\pi_{1} = (3+4\zeta^{2}+4\zeta^{3})$,
$\pi_{2}= (-1-4\zeta^{2} -4 \zeta^{3})$.
Notice that in this case, $\zeta \notin N_{K/F}(K^{\ast})$. We observe that,
$$ -\zeta^{2}(1+\zeta) \equiv 3^5 \pmod{ \pi_{1}},$$
$$ -\zeta^{2}(1+\zeta) \equiv 6^5 \pmod{ \pi_{2}}. $$
Hence $q^{\ast}=1$ and $t=g-1=1$. So $1 \leq \text{{\rm rank} }S_{K}
\leq 4$ by our result. This is confirmed by SAGE - see table 1.
\end{example}
\vskip 5mm

\begin{example}
Let $K= \Q(\zeta,42^{\frac{1}{5}})$. Notice that in this case $\zeta
\notin N_{K/F}(K^{\ast})$. From examples 5.2 and 5.3, we see that
$\zeta^2(1+\zeta) \in N_{K/F}(K^{\ast})$.\\
Hence $q^{\ast}=1$ and $t=g-1=2$. Thus, $2 \leq \text{{\rm rank}}
S_{K} \leq 8$ from our result. This is again confirmed by a SAGE
program as seen from table 1.
\end{example}
\vskip 5mm

\subsection{Constructing genus fields}
\vskip 3mm

\noindent We want to find elements $x_{1},\cdots,x_{t}\in K$ such
that the genus field
$M_1=K(x_{1}^{\frac{1}{5}},\cdots,x_{t}^{\frac{1}{5}})$.\\
{\it In the following proposition, we restrict our attention only to
those elements $x$ for which each $\pi$ that divides $x$ is of the
form, $\pi \equiv a \pmod{ 5\Z[\zeta]}$ for some $a \in \{1,2,3,4\}$.}
\vskip 5mm

\begin{prop}\label{generators}
Let $F=\Q(\zeta)$, and let $K=F(x^{\frac{1}{5}})$ be cyclic of
degree $5$ as above. Writing
$$x=u \lambda^{e_{\lambda}}  \pi_{1}^{e_{1}} \cdots \pi_{f}^{e_{f}}
\pi_{f+1}^{e_{f+1}} \cdots \pi_{g}^{e_{g}},$$ where each $\pi_{i}
\equiv \pm 1, \pm 7 \pmod{ \lambda^{5}}$, for $1\leq i \leq f$ and  $\pi_{j} \not\equiv
\pm 1, \pm 7 \pmod{ \lambda^{5}}$,
for $f+1 \leq j \leq g$.\\ Then, we have:\\
(i) there exist $h_{i} \in \{1,2,3,4\}$ such that
$\pi_{f+1}\pi_{i}^{h_{i}} \equiv \pm1, \pm7\pmod{ \lambda^{5}}$, for $f+2 \leq i \leq
g$;\\
(ii) if $(\lambda)$ ramifies in $K/F$ and each $\pi_k|x$ has the
property that some, but not all $\zeta^{i}(1+\zeta)^{j}$ are $5$-th
powers modulo $(\pi_{k}))$ in $\mathbb{Z}[\zeta]$, then the genus
field $M_{1}$ is given as
\begin{equation}\label{genusfield}
M_{1}=K(\pi_{1}^{\frac{1}{5}},\cdots,\pi_{f}^{\frac{1}{5}},
(\pi_{f+1}\pi_{f+2}^{h_{f+2}})^{\frac{1}{5}},\cdots,
(\pi_{f+1}\pi_{g}^{h_{g}})^{\frac{1}{5}})
\end{equation}
where $h_{i} \in \{1,2,3,4\}$ is chosen as in (i);\\
(iii) in the other cases, the the genus field $M_{1}$ is given
similarly by deleting an appropriate number of $5^{th}$ roots from
the right-hand side of the equation (\ref{genusfield}).
\end{prop}

\begin{proof}
The proof of (i) is straightforward and we proceed to prove (ii).\\
Suppose first that $(\lambda)$ ramifies in $K/F$ and that for each
$\pi_k|x$, some (but not all) $\zeta^{i}(1+\zeta)^{j}$ are $5$-th
powers modulo $(\pi_{k}))$ in $\Z[\zeta]$. Let $M_{1}^{\prime}$
denote the field given on the right-hand side of equation
(\ref{genusfield}); we shall prove that $M_{1}^{\prime}$ is the
genus field $M_1$ of degree $5^t$ over $K$ which corresponds to
$S_K/\lambda S_K$. Note that, by the previous proposition, the
number of $5^{th}$ roots in this expression is $t$. Next, we note
that only $\pi_{i}$ ramifies in $F(\pi_{i}^{\frac{1}{5}})/F$ for
$1\leq i \leq f$, and that only the primes $\pi_{i}$ and $\pi_{f+1}$
ramify in $F((\pi_{f+1}\pi_{i}^{h_{i}})^{\frac{1}{5}})/F$ for $f+2
\leq i \leq g$. Hence, each of the fields
$$F(\pi_{1}^{\frac{1}{5}}),\cdots,F(\pi_{f}^{\frac{1}{5}}),
F((\pi_{f+1}\pi_{f+2}^{h_{f+2}})^{\frac{1}{5}}),
\cdots,F((\pi_{f+1}\pi_{g}^{h_{g}})^{\frac{1}{5}}),F(x^{\frac{1}{5}})$$
is linearly disjoint from the composite of the other fields. Thus,
$$ [F(\pi_{1}^{\frac{1}{5}},\cdots,\pi_{f}^{\frac{1}{5}},
(\pi_{f+1}\pi_{f+2}^{h_{f+2}})^{\frac{1}{5}},\cdots,(\pi_{f+1}
\pi_{g}^{h_{g}})^{\frac{1}{5}},x^{\frac{1}{5}}) : F ] =
5^{t+1}.$$ This implies $[M_{1}^{\prime}:K] =5^{t}$. As
$[M_{1}:K]= 5^{t}$, we will have $M_{1}^{\prime} = M_{1}$ if we
can show that $M_{1}^{\prime} \subset M_{1}$. Now, by definition,
$M_{1}$ is the maximal abelian extension of $F$ contained in the
Hilbert class field of $K$. Since $M_{1}^{\prime}$ is a composite of
linearly disjoint abelian extensions, it is an abelian extension of
$F$. Therefore, to show $M_{1}^{\prime} = M_{1}$, it suffices to
show that $M_{1}^{\prime}$ is unramified over $K$. But, this is true
because each $(\pi_{i})$ is $5^{th}$ power of an ideal in $K$,
$\pi_{i} \equiv \pm1,\pm7 \pmod{ \lambda^{5}}$ for $1\leq i \leq f$ and
$\pi_{f+1}\pi_{i}^{h_{i}} \equiv \pm1,\pm 7 \pmod{ \lambda^{5}}$ for $f+2 \leq i \leq g$.\\
Thus, we have proved (ii). \\
The remaining case (iii) is handled completely similarly; we just
need to delete an appropriate number of any $5^{th}$ roots from the
right-hand side of the equation (\ref{genusfield}).
\end{proof}
\vskip 5mm

\begin{cor}
Let $F,K,x$ be as in the proposition. Further, suppose the genus
field $M_{1}$ of $K/F$ is described as in the proposition. Then, for
$i=1,2$, the genus fields $M_{i+1}$ are obtained
recursively by deleting $s_{i}$ generators of the field $M_{i}$.
\end{cor}

\begin{proof}
We have $[M_{i}:K] = 5^{t-s_{1}-\cdots - s_{i-1}}$ and
$[M_{i+1}:K]= 5^{t-s_{1}-\cdots - s_{i}}$. Since each of the
fields $F(\pi_{1}^{\frac{1}{5}}),
\cdots,F(\pi_{f}^{\frac{1}{5}}),F((\pi_{f+1}\pi_{f+2}^{h_{f+2}})^{\frac{1}{5}}),
\cdots,F((\pi_{f+1}\pi_{g}^{h_{g}})^{\frac{1}{5}}),F(x^{\frac{1}{5}})$
is linearly disjoint from the composite of the other fields, the
result follows.
\end{proof}

\vskip 5mm

\noindent Now, we look for representative ideals
$\mathfrak{a}_{1},\cdots,\mathfrak{a}_{t}$ whose classes form a
basis of the ambiguous ideal class group $S_{K}[\lambda]$.
Similarly, we also look for representative ideals
$\mathfrak{b}_{1},\cdots,\mathfrak{b}_{t-s_{1}-\cdots-s_{i}}$ whose
classes form a basis of $(\lambda^i S_{K})[\lambda]$ for $i=1,2$.
For this purpose, we find ideals whose classes generate
$S_{K}^{(\tau),i}$ for $i=1,2,3$. \\
We observe that the ambiguous ideal class group $S_k[\lambda]$ may
be identified with the group $S_{K,s}$ of strongly ambiguous ideal
classes, excepting the case when at least one of
$\zeta^{i}(1+\zeta)^{j} \in N_{K/F}(K^{\ast})$, and $\zeta \not\in
N_{K/F}(E_{K})$,
where $E_{K}$ is the group of units of $K$. \\

\noindent We note that a necessary condition for the exceptional
case to occur is that for any $\pi_{k}|x$, one has
$\zeta^{i}(1+\zeta)^{j} \equiv a^{5} \pmod{ (\pi_{k})}$ for some
$a\in \Z[\zeta]$ and some $i,j$. There are two possible situations
when the exceptional case occurs. Namely, if both $\zeta, 1+\zeta$
are norms of elements from $K^{\ast}$, but neither of them is a norm
from $E_{K}$, then $S_{K}[\lambda]$ is the direct product of
$S_{K,s}$ and two cyclic groups of order $5$. In other exceptional
cases, $S_{K}[\lambda]$ is the direct product of $S_{K,s}$ and a
cyclic group of order $5$. \vskip 5mm

\subsection{Using ideles to express in terms of Hilbert symbols}
\vskip 3mm

\noindent We saw in section 4 how to obtain matrices with entries in
$\mathbb{F}_l$ whose ranks are equal to the $s_i$'s. In this
section, where $l=5$ and the genus fields are chosen as above, we
explain what these matrices simplify to.\\
Using the notation of Proposition \ref{generators}, we choose prime
ideals $\mathfrak{B}_{i}$ in $K$ such that $5 \mathfrak{B}_{i}=
(\pi_{i})$ for $1\leq i \leq g$. If $(\lambda)$ ramifies in $K/F$,
we let $\mathfrak{I}$ denote the prime ideal in $K$ such that
$5 \mathfrak{I}= (\lambda)$. If there exists ambiguous ideal
classes of $K/F$ which are not strongly-ambiguous, we let
$\mathfrak{B}$ be a prime ideal which is contained in one such class
and is relatively prime to $x_{1},\cdots,x_{t}$, where
$M_{1}=K(x_{1}^{\frac{1}{5}},\cdots,x_{t}^{\frac{1}{5}})$ and
$x_{1},\cdots,x_{t} \in F$. If $q^{\ast}=2$ and $q=0$, we choose
$\mathfrak{B^{\prime}}$ to be a prime ideal contained in another
class (from $\mathfrak{B}$) of ideal which is ambiguous
but not strongly-ambiguous, and is relatively prime to $\mathfrak{B}$,$x_{1},\cdots,x_{t}$.  \\
Let $I_{K}^{(\tau)}$ denote the free abelian group generated by
these prime ideals. In other words, $I_K^{(\tau)}$ is generated by
$\mathfrak{B}_{1},\cdots,\mathfrak{B}_{g},$ and $\mathfrak{I}$ (in
case $(\lambda)$ ramifies in $K/F$), and $\mathfrak{B}$ (in the case
when there exist ambiguous ideal classes which are not
strong-ambiguous), and also $\mathfrak{B^{\prime}}$ (in case $q^{\ast}=2$ and $q=0$). \\
Let $D_{K}^{(\tau)} = I_{K}^{(\tau)}/5 I_{K}^{(\tau)}$. Viewed
as a vector space over $\mathbb{F}_{5}$, let $D_{K}^{(\tau)}$ have
dimension $u$. Then $u=g,g+1,g+2$ or $g+3$ in the four possibilities
mentioned above respectively. Now, the map $I_{K}^{(\tau)} \to
S_{K}[\lambda]$ sending each ideal to its ideal class induces
surjective homomorphisms
$$\omega_{1}: D_{K}^{(\tau)} = I_{K}^{(\tau)}/5 I_{K}^{(\tau)}
\to S_{K}[\lambda].$$

\noindent Recall the map $\phi_{1} : S_K[\lambda] \to
\mathbb{F}_5^t$ constructed in the proof of Theorem \ref{thm2}.
Define $\eta_{1}:= \phi_{1}\circ \omega_{1}: D_{K}^{(\tau)} \to
\mathbb{F}_{5}^{t}$.\\
For $1\leq i \leq t, 1\leq j \leq g,$ let
$\mu_{ij}$ denote the Artin symbol
$\bigg(\frac{K(x_{i}^{\frac{1}{5}})/K}{\mathfrak{B}_{j}^{e_{j}}}\bigg).$
Further, suppose
$$\mu_{i(g+1)} = \bigg(\frac{K(x_{i}^{\frac{1}{5}})/K}{\mathfrak{I}}\bigg)
\text {   for } 1\leq i \leq t, \text{ if } (\lambda) \text{
ramifies in }K/F,$$ and
$$\mu_{i(g+2)} = \bigg(\frac{K(x_{i}^{\frac{1}{5}})/K}{\mathfrak{B}}\bigg) \text {   for }
1\leq i \leq t, \text{ if}~ S_K[\lambda] \setminus S_{K,s} \neq
\emptyset,$$ and
$$\mu_{iu} = \bigg(\frac{K(x_{i}^{\frac{1}{5}})/K}{\mathfrak{B^{\prime}}}\bigg) \text {   for }
1\leq i \leq t, \text{ if}~ |S_K[\lambda]/S_{K,s}| > 5.$$
 If $\gamma_{ij}\in \mathbb{F}_{\ell}$ are defined by
the power symbol $\zeta^{\gamma_{ij}} =
(x_{i}^{\frac{1}{\ell}})^{\mu_{ij}-1}$, let $C_{1}$ be the matrix
$(\gamma_{ij}), 1\leq i \leq t, 1\leq j \leq u$.
It is clear that $C_{1}$ is the matrix of $\eta_{1}$ with respect to
the ordered basis
$\{\mathfrak{B}_{1}^{e_{1}},\cdots,\mathfrak{B}_{g}^{e_{g}},\mathfrak{I}$
(if included),$\mathfrak{B}$ (if included), and $\mathfrak{B^{\prime}}$ (if included)\}. Since $\omega_{1}$
is surjective, ${\rm rank} C_{1}= {\rm rank} A_{1}= s_{1}$ (see Theorem
\ref{thm2}).\\

\noindent We next construct ideles
$a_{\mathfrak{B_{1}}},\cdots,a_{\mathfrak{B_{g}}},a_{\mathfrak{I}},a_{\mathfrak{B}},a_{\mathfrak{B^{\prime}}}
\in J_{K}$, the idele group of $K$, such that
$$(a_{\mathfrak{B_{j}}}, K(x_{i}^{\frac{1}{5}})/K) = \bigg(\frac{K(x_{i}^{\frac{1}{5}})/K}{\mathfrak{B}_{j}^{e_{j}}}\bigg) \text{  for }1\leq i \leq t, 1\leq j \leq g,$$
$$(a_{\mathfrak{I}}, K(x_{i}^{\frac{1}{5}})/K) = \bigg(\frac{K(x_{i}^{\frac{1}{5}})/K}{\mathfrak{I}}\bigg) \text{  for }1\leq i \leq t, $$
$$(a_{\mathfrak{B}}, K(x_{i}^{\frac{1}{5}})/K) = \bigg(\frac{K(x_{i}^{\frac{1}{5}})/K}{\mathfrak{B}}\bigg) \text{  for }1\leq i \leq t, $$
$$(a_{\mathfrak{B^{\prime}}}, K(x_{i}^{\frac{1}{5}})/K) = \bigg(\frac{K(x_{i}^{\frac{1}{5}})/K}{\mathfrak{B^{\prime}}}\bigg) \text{  for }1\leq i \leq t. $$
This is done as follows. Let
$$a_{\mathfrak{B_{j}}} =(\cdots,1,x^{\frac{1}{5}},1,\cdots) \text{   for } 1\leq j \leq g,$$
the idele which is 1 at all places except at the place corresponding to $\mathfrak{B_{j}}$, where it is $x^{\frac{1}{5}}$. Let
$$a_{\mathfrak{I}} =(\cdots,1,x_{\mathfrak{I}},1,\cdots),$$
the idele which is 1 at all places except at the place corresponding to $\mathfrak{I}$, where we insert an element $x_{\mathfrak{I}} \in K$, such that $\mathfrak{I}|x_{\mathfrak{I}}$, but $\mathfrak{I}^2 \nmid x_{\mathfrak{I}}$.Let
$$a_{\mathfrak{B}} =(\cdots,1,x_{\mathfrak{B}},1,\cdots),$$
the idele which is 1 at all places except at the place corresponding to $\mathfrak{B}$, where we insert an element $x_{\mathfrak{B}} \in K$, such that $\mathfrak{B}|x_{\mathfrak{B}}$, but $\mathfrak{B}^2 \nmid x_{\mathfrak{B}}$.Let
$$a_{\mathfrak{B^{\prime}}} =(\cdots,1,x_{\mathfrak{B^{\prime}}},1,\cdots),$$
the idele which is 1 at all places except at the place corresponding to $\mathfrak{B^{\prime}}$, where we insert an element $x_{\mathfrak{B^{\prime}}} \in K$, such that $\mathfrak{B^{\prime}}|x_{\mathfrak{B^{\prime}}}$, but $\mathfrak{B^{\prime}}^2 \nmid x_{\mathfrak{B^{\prime}}}$.\\
Now
$$(a_{\mathfrak{B_{j}}}, K(x_{i}^{\frac{1}{5}})/K)|F(x_{i}^{\frac{1}{5}}) = (N_{K/F}(a_{\mathfrak{B_{j}}}), F(x_{i}^{\frac{1}{5}})/F),$$
where $N_{K/F}(a_{\mathfrak{B_{j}}})$ is the idele $(\cdots,1,x,1,\cdots)$ of $F$ which is 1 at all places except at the place corresponding to $(\pi_{j})$, where it is $x$. We denote $N_{K/F}(a_{\mathfrak{B_{j}}})$ by $a_{\pi_{j}}$. Similarly,
$$(a_{\mathfrak{I}}, K(x_{i}^{\frac{1}{5}})/K)|F(x_{i}^{\frac{1}{5}}) = (a_{\lambda}, F(x_{i}^{\frac{1}{5}})/F),$$
where $a_{\lambda} = N_{K/F}(a_{\mathfrak{I}})=(\cdots,1,x_{\lambda},1,\cdots)$ with $x_{\lambda}=N_{K/F}(x_{\mathfrak{I}})$. Also,
$$(a_{\mathfrak{B}}, K(x_{i}^{\frac{1}{5}})/K)|F(x_{i}^{\frac{1}{5}}) = (a_{\pi}, F(x_{i}^{\frac{1}{5}})/F),$$
where $a_{\pi} = N_{K/F}(a_{\mathfrak{B}})=(\cdots,1,x_{\pi},1,\cdots)$, where $\pi = N_{K/F}(\mathfrak{B})$, with $x_{\pi}=N_{K/F}(x_{\mathfrak{B}})$. And finally,
$$(a_{\mathfrak{B^{\prime}}}, K(x_{i}^{\frac{1}{5}})/K)|F(x_{i}^{\frac{1}{5}}) = (a_{\pi^{\prime}}, F(x_{i}^{\frac{1}{5}})/F),$$
where $a_{\pi^{\prime}} = N_{K/F}(a_{\mathfrak{B^{\prime}}})=(\cdots,1,x_{\pi^{\prime}},1,\cdots)$, where $\pi^{\prime} = N_{K/F}(\mathfrak{B^{\prime}})$, with $x_{\pi^{\prime}}=N_{K/F}(x_{\mathfrak{B^{\prime}}})$.\\

\noindent We now consider $\zeta^{\gamma_{ij}} =
(x_{i}^{\frac{1}{5}})^{\mu_{ij}-1}$. From our calculation we can
replace
$$ \mu_{ij} \text {  by  }  \nu_{ij} = (a_{\pi_{j}},F(x_{i}^{\frac{1}{5}})/F) \text{  for } 1\leq i \leq t , 1\leq j \leq g,$$
$$ \mu_{i(g+1)} \text {  by  }  \nu_{i(g+1)} = (a_{\lambda},F(x_{i}^{\frac{1}{5}})/F) \text{  for } 1\leq i \leq t , $$
$$ \mu_{i(g+2)} \text {  by  }  \nu_{i(g+2)} = (a_{\pi},F(x_{i}^{\frac{1}{5}})/F) \text{  for } 1\leq i \leq t , $$
$$ \mu_{i(g+3)} \text {  by  }  \nu_{i(g+3)} = (a_{\pi^{\prime}},F(x_{i}^{\frac{1}{5}})/F) \text{  for } 1\leq i \leq t. $$
So we have,
$$\zeta^{\gamma_{ij}} = (x_{i}^{\frac{1}{5}})^{\nu_{ij}-1} \text{  for all }i,j.$$
Since the ideles $a_{\pi_{j}} (1\leq j \leq g), a_{\lambda}, a_{\pi}, a_{\pi^{\prime}}$ are local ideles, we may identify the expressions $(x_{i}^{\frac{1}{5}})^{\nu_{ij}-1}$ with the degree $5$ Hilbert symbols $\Big(\frac{x_{i},x}{(\pi_{j})}\Big),\Big(\frac{x_{i},x_{\lambda}}{(\lambda)}\Big),\Big(\frac{x_{i},x_{\pi}}{(\pi)}\Big),\Big(\frac{x_{i},x_{\pi^{\prime}}}{(\pi^{\prime})}\Big)$ for the local fields $F_{\pi_{j}}(x_{i}^{\frac{1}{5}})/F_{\pi_{j}}, F_{\lambda}(x_{i}^{\frac{1}{5}})/F_{\lambda},F_{\pi}(x_{i}^{\frac{1}{5}})/F_{\pi},F_{\pi^{\prime}}(x_{i}^{\frac{1}{5}})/F_{\pi^{\prime}}$ respectively.\\

\noindent
Finally we want to simplify, $\Big(\frac{x_{i},x_{\lambda}}{(\lambda)}\Big),\Big(\frac{x_{i},x_{\pi}}{(\pi)}\Big)$ and $\Big(\frac{x_{i},x_{\pi^{\prime}}}{(\pi^{\prime})}\Big)$.\\
We may write $x_{\lambda} = \lambda y_{\lambda} z_{\lambda}^{-1}$,
where $y_{\lambda},z_{\lambda}$ are integers in $F$, congruent to
$\pm1,\pm 2 \pmod{ \lambda}$. Since $x_{i} \equiv \pm 1,\pm 7 \pmod{
\lambda^{5}}$, let $\alpha^{5}=x_{i}$, and write $\alpha= \pm 1+
\lambda y$ or $\alpha = \pm 2 + \lambda y$ (respectively).  Since
$y$ is a root of a polynomial $f(Y) \in
\mathcal{O}_{F_{\lambda}}[Y]$, such that $f(Y)\equiv Y^{5}-Y-c \pmod
{\lambda}$, we have $f^{\prime}(y) \equiv -1 \neq 0 \pmod{ \lambda}$.
Thus $F_{\lambda}(y) = F_{\lambda}(x_{i}^{\frac{1}{5}})$ is
unramified over $F_{\lambda}$. Thus we have,
$\Big(\frac{x_{i},y_{\lambda}}{(\lambda)}\Big) =
\Big(\frac{x_{i},z_{\lambda}}{(\lambda)}\Big) =1$ (See,
\cite{Serre}[page 209, Exercise 5]). So,
$$\Big(\frac{x_{i},x_{\lambda}}{(\lambda)}\Big)=\Big(\frac{x_{i},\lambda}{(\lambda)}\Big)\Big(\frac{x_{i},y_{\lambda}}{(\lambda)}\Big)\Big(\frac{x_{i},z_{\lambda}}{(\lambda)}\Big)^{-1} =\Big(\frac{x_{i},\lambda}{(\lambda)}\Big) , \text{  for } 1\leq i \leq t. $$

\noindent Now, we write $x_{\pi}= \pi y_{\pi}$, where $y_{\pi}$ is
relatively prime to $\pi$. Since $\mathfrak{B}$ was chosen
relatively prime to $x_{1},\cdots,x_{t}$, then $\pi$ is relatively
prime to $x_{i}$ for all $i$. Hence
$$\Big(\frac{x_{i},x_{\pi}}{(\pi)}\Big)=\Big(\frac{x_{i},\pi}{(\pi)}\Big)\Big(\frac{x_{i},y_{\pi}}{(\pi)}\Big) =\Big(\frac{x_{i},\pi}{(\pi)}\Big) , \text{  for } 1\leq i \leq t. $$

\noindent Finally, let us write $x_{\pi^{\prime}}= \pi^{\prime}
y_{\pi^{\prime}}$, where $y_{\pi^{\prime}}$ is relatively prime to
$\pi^{\prime}$. Since $\mathfrak{B^{\prime}}$ was chosen relatively
prime to $x_{1},\cdots,x_{t}$, then $\pi^{\prime}$ is relatively
prime to $x_{i}$ for all $i$. Hence
$$\Big(\frac{x_{i},x_{\pi^{\prime}}}{(\pi^{\prime})}\Big)=\Big(\frac{x_{i},\pi^{\prime}}{(\pi^{\prime})}\Big)\Big(\frac{x_{i},y_{\pi^{\prime}}}{(\pi^{\prime})}\Big) =\Big(\frac{x_{i},\pi^{\prime}}{(\pi^{\prime})}\Big) , \text{  for } 1\leq i \leq t. $$
With these notations, we may describe the matrix whose entries are
power residue symbols and, whose {\rm rank} gives us the {\rm rank}
of the piece $H_1$ (see \ref{prop1}) of the $l$-class group. \vskip
5mm

\begin{thm}
Let $F=\Q(\zeta)$, $K=F(x^{\frac{1}{5}})$, $x= u
\lambda^{e_{\lambda}}  \pi_{1}^{e_{1}} \cdots \pi_{g}^{e_{g}}$ as
above. Let $M_{1}=K(x_{1}^{\frac{1}{5}},\cdots,x_{t}^{\frac{1}{5}})$
denote the genus field of $K/F$, where $[M_{1}:K]=5^{t}$, $x_{i} \in
F$ for $1\leq i \leq t$, and $x_{i} \equiv \pm1,\pm7 \pmod{
\lambda^{5}}$. Let $\mathfrak{B},\mathfrak{B^{\prime}}$ be ideals as
above  defined respectively when there exist ambiguous ideal classes
which are not strongly-ambiguous, and when $q^{\ast}=2,q=0$. Let
$(\pi)=N_{K/F}(\mathfrak{B})$ and $(\pi^{\prime}) =
N_{K/F}(\mathfrak{B^{\prime}})$, where $N_{K/F}$ is the norm map
from $K$ to $F$. For $1\leq i \leq t, 1\leq j \leq g,$ let
$\nu_{ij}$ denote the degree $5$ Hilbert symbol
$\Big(\frac{x_{i},x}{(\pi_{j})}\Big)$. Further, suppose
$$\nu_{i(g+1)} = \Big(\frac{x_{i},\lambda}{(\lambda)}\Big)
\text {   for } 1\leq i \leq t, \text{ if } (\lambda) \text{
ramifies in }K/F,$$ and
$$\nu_{i(g+2)} = \Big(\frac{x_{i},\pi}{(\pi)}\Big) \text {   for }
1\leq i \leq t, \text{ if}~ S_K^{(\tau)} \setminus S_{K,s}^{(\tau)} \neq
\emptyset,$$ and
$$\nu_{iu} = \Big(\frac{x_{i},\pi^{\prime}}{(\pi^{\prime})}\Big) \text {   for }
1\leq i \leq t, \text{ if}~ |S_K^{(\tau)} /S_{K,s}^{(\tau)}| > 5.$$
 If $\gamma_{ij}\in \mathbb{F}_{\ell}$ are defined by
the power symbol $\zeta^{\gamma_{ij}} =
(x_{i}^{\frac{1}{\ell}})^{\nu_{ij}-1}$, and $C_{1}$ is the matrix
$(\gamma_{ij}), 1\leq i \leq t, 1\leq j \leq u$, we have
$$s_{1}= {\rm rank} H_1 =
 {\rm rank} C_{1}.$$
\end{thm}
\vskip 5mm

\noindent Finally, we discuss how the above theorem can be
generalized to determine the ranks $s_{i}$'s for $i>1$. Observe that
$$S_{K}[\lambda] \supset (\lambda S_K)[\lambda] \supset (\lambda^2
S_{K})[\lambda].$$ Since
$cl(\mathfrak{B}_{1}),\cdots,cl(\mathfrak{B}_{g}),cl(\mathfrak{I})$
(if included), $cl(\mathfrak{B})$ (if included), and
$\mathfrak{B^{\prime}}$ (if included) generate $S_{K}[\lambda]$,
there exists a basis of $(\lambda^{i-1}S_{K})[\lambda]$, $\{
cl(\Gamma_{i,1}),\cdots, cl(\Gamma_{i,t-s_{1}-\cdots-s_{i-1}})\}$
consisting of elements which are $\mathbb{F}_{5}$-linear
combinations of
$cl(\mathfrak{B}_{1}),\cdots,cl(\mathfrak{B}_{g}),cl(\mathfrak{I})$,
$cl(\mathfrak{B}),cl(\mathfrak{B^{\prime}})$, for $i=2,3$. \\
Let $\Gamma_{i,1},\cdots,\Gamma_{i,t-s_{1}-\cdots-s_{i-1}}$ be some
representative ideals for the respective classes. With these choices,
we have the following theorem expressing the ranks $s_i$ in terms of
matrices over $\mathbb{F}_{\ell}$:\\

\begin{thm}
Let $F=\Q(\zeta)$, $K=F(x^{\frac{1}{5}})$, where $x=u
\lambda^{e_{\lambda}} \pi_{1}^{e_{1}} \cdots
\pi_{g}^{e_{g}}$ as above. Let $M_{1}$ be the genus field of $K/F$
and, for $i=1,2$, let $M_{i+1}=
K(y_{1}^{\frac{1}{5}},\cdots,y_{t-s_{1}-\cdots-s_{i}}^{\frac{1}{5}})$,
as in Theorem\ref{thm3}. Let $\Gamma_{i+1,1},\cdots,\Gamma_{i+1,t-s_{1}-\cdots-s_{i}}$
 be as in the previous paragraph.
 Denote
$$\mu_{jk} = \bigg( \frac{K(y_{j}^{\frac{1}{5}})/K}{\Gamma_{i+1,k}}\bigg) \text {   for } 1\leq j,k \leq t-s_{1}-\cdots-s_{i}.$$
 If
$\gamma_{jk}$ are defined by $\zeta^{\gamma_{jk}} =
(y_{j}^{\frac{1}{\ell}})^{\mu_{jk}-1}$, and $C_{i+1}=(\gamma_{jk}),
1\leq j,k \leq t-s_{1}-\cdots -s_{i}$, then
$$s_{i+1}= {\rm rank} H_{i+1}= {\rm rank}  C_{i+1}.$$

\end{thm}

\begin{proof}
We have the map $$\phi_{i+1}: (\lambda^i S_{K})[\lambda] \to
\mathbb{F}_{5}^{t-s_{1}-\cdots-s_{i}}$$ constructed in the proof of
Theorem \ref{thm3}. Clearly, $C_{i+1}$ is the matrix of $\phi_{i+1}$
with respect to the ordered basis
$$\{ \Gamma_{i+1,1},\cdots,\Gamma_{i+1,t-s_{1}-\cdots-s_{i}}
\}.$$ Thus, from Theorem \ref{thm3},
$$s_{i+1} = {\rm rank} C_{i+1}.$$
\end{proof}

\vskip 5mm

\noindent{\bf Remarks.} In conclusion, the above theorems show in
principle how to compute $t,s_{1},s_{2},s_{3}$. We can use them to
find the {\rm rank} of $S_{K}$ using the formula obtained from Theorem
\ref{thm1}, namely,
$${\rm rank} S_{K} = 4t-3s_{1}-2s_{2}-s_{3}.$$
However, concrete determination of $s_2,s_3$ seems to be difficult.
In particular, it would be useful to find explicit generators for
the groups $(\lambda^i S_{K})[\lambda]$ for $i \geq 1$. We also
obtain a bound for the {\rm rank} of $S_{K}$ in terms of $t$ and
$s_{1}$ as follows,
$$2t-s_{1} \leq {\rm rank} S_{K} \leq 4t-3s_{1}.$$
\vskip 5mm

\subsection{Applications - explicit results}

\noindent We apply our result in various situations to give sharp
bounds for the rank of the $5$-class group.

\begin{thm}\label{-7mod25}
Let $p_{i} \equiv \pm 7 \pmod{25}$ for $1 \leq i \leq r$ be primes
and $r \geq 2$. Let $n=p_{1}^{a_{1}} \cdots p_{r}^{a_{r}}$, where $1
\leq a_{i} \leq 4$ for $1 \leq i \leq r$. Let $F=\Q(\zeta_{5})$ and
$K=F(n^{\frac{1}{5}})$. Assume that all ambiguous ideal classes of
$K/F$ are strongly ambiguous.
Then, the $\lambda^2$-rank of $S_{K}$ is $r-1$ and $2r-2 \leq \mathrm{rank } S_{K} \leq 4r-4$. \\
If there are ambiguous ideal classes which are not strongly
ambiguous, then $s_{1} \leq 2$, and the $\lambda^{2}$-rank of
$S_{K}$ is greater than or equal to $r-3$ and $\max(2r-4,r-1) \leq
\mathrm{rank } S_{K} \leq 4r-4$.
\end{thm}

\begin{proof}
Firstly we notice that $n \equiv \pm 1, \pm 7 \pmod{25}$. So $\lambda$ does not ramify in $K/F$. Looking at the fields $K_{i}=F(p_{i}^{\frac{1}{5}})$, one can easily see that $\zeta$ and
$1+\zeta$ are fifth powers modulo $p_{i}$ for all $i=1,\cdots,r$. Thus $q^{\ast}=q=2$ and $t=d-3+q^{\ast}=r-1$.\\
To compute $s_{1}$, let $x_{i}=p_{i}$ where $1 \leq i \leq r-1$. Using \cite{Serre}[Chapter 14, Section 3]  one can easily check that
$\Big(\frac{x_i,n}{p_{j}}\Big) = 1$ for $1 \leq i \leq r-1$ and $1 \leq j \leq r$. That is the $(r-1) \times r$ matrix $C_{1}$ is
the zero matrix. So, $s_{1}=0$.\\
Thus, we get $\lambda^2$-rank of $S_{K}$ is $t-s_{1}=r-1$. Since $2t-s_{1} \leq \mathrm{rank }S_{K} \leq 4t-3s_{1}$,
we obtain that $2r-2 \leq \mathrm{rank } S_{K} \leq 4r-4$.\\
the second part of the statement follows from the fact that $0 \leq
q \leq 1$, and then the matrix $C_{1}$ is of size $(r-1) \times
(r+q^{\ast}-q)$, which can have rank at most $2$.
\end{proof}

\begin{thm}\label{not-7mod25}
Let $p_{i} \equiv \pm 7 \pmod{25}$ for $1 \leq i \leq r$ be primes
and let $q_{j}$ be primes such that $q_{j} \equiv \pm 2 \pmod{5}$
but $q_{j} \not\equiv \pm7 \pmod{25}$ for $1 \leq j \leq s$. Let
$n=p_{1}^{a_{1}} \cdots p_{r}^{a_{r}}q_{1}^{b_{1}}\cdots
q_{s}^{b_{s}}$, where $1 \leq a_{i},b_{j} \leq 4$ for $1 \leq i \leq
r$ and $1 \leq j \leq s$. Let $n \not\equiv \pm 1, \pm 7 \pmod{25}$.
Let $F=\Q(\zeta_{5})$ and $K=F(n^{\frac{1}{5}})$. Assume that all
ambiguous ideal classes of $K/F$ are strongly ambiguous.
Then, the $\lambda^2$-rank of $S_{K}$ is $r+s-1$ and $2r+2s-2 \leq \mathrm{rank } S_{K} \leq 4r+4s-4$.\\
If there are ambiguous ideal classes which are not strongly
ambiguous, then $s_{1} \leq 1$, and the $\lambda^{2}$-rank of
$S_{K}$ is greater than or equal to $r+s-2$ and $\max(2r+2s-3,r+s-1)
\leq \mathrm{rank } S_{K} \leq 4r+4s-4$.
\end{thm}

\begin{proof}
Firstly we notice that $\lambda$ ramifies in $K/F$. Since $N(q) \not\equiv 1 \pmod{25}$, $\zeta \notin N_{K/F}(K^{\ast})$. Looking at the fields $K_{i}=F(p_{i}^{\frac{1}{5}})$ and $L_{1}=F(q_{1}^{\frac{1}{5}})$, $L_{j}=F((q_{1}q_{j}^{h_{j}})^{\frac{1}{5}}$, where $1\leq h_{j} \leq 4$ are chosen such that $q_{1}q_{j}^{h_{j}} \equiv \pm 1,\pm 7 \pmod{25}$, $j \neq 1$, one can easily see that  some $\zeta^{i}(1+\zeta)^{j}$ is fifth power modulo $p_{i}$ for all $i=1,\cdots,r$ and $q_{j}$ for $1 \leq j \leq s$. Thus $q^{\ast}=q=1$ and $t=d-3+q^{\ast}=r+s-1$.\\
To compute $s_{1}$, let $x_{i}=p_{i}$ where $1 \leq i \leq r$ and $y_{j-1} = q_{1}q_{j}^{h_{j}}$ where $2 \leq j \leq s$.
Using \cite{Serre}[Chapter 14, Section 3]  one can easily check that $\Big(\frac{x_i,n}{p_{j}}\Big) = 1$ for
$1 \leq i,j \leq r$, $\Big(\frac{x_i,n}{q_{j}}\Big) = 1$ for $1 \leq i \leq r$, $1\leq j \leq s$,
$\Big(\frac{y_i,n}{p_{j}}\Big) = 1$ for $1 \leq i \leq s-1$, $1\leq j \leq r$ and $\Big(\frac{y_i,n}{q_{j}}\Big) = 1$ for
$1 \leq i \leq s-1$, $1\leq j \leq s$.
That is, the $(r +s-1) \times r+s$ sub matrix of $C_{1}$ is zero matrix. Since $x_{i},y_{i} \equiv \pm 7 \pmod{25}$,
using \cite{Cassels}[Exercise 2.12,pg353-354] one can easily check that
$\Big(\frac{x_{i}, \lambda}{\lambda}\Big) \Big(\frac{y_{i}, \lambda}{\lambda}\Big)= 1$. Therefore, $s_{1}=0$.\\
So, we see that the $\lambda^2$-rank of $S_{K}$ is $t-s_{1}=r+s-1$. Since $2t-s_{1} \leq \mathrm{rank }S_{K} \leq 4t-3s_{1}$, we obtain that $2r+2s-2 \leq \mathrm{rank } S_{K} \leq 4r+4s-4$.\\
The second part of the statement follows from the fact that $q=0$,
as, then the matrix $C_{1}$ is of size $(r+s-1) \times (r+s+2)$,
which can have rank at most $1$.
\end{proof}

\begin{thm}\label{is-7mod25}
Let $p_{i} \equiv \pm 7 \pmod{25}$ for $1 \leq i \leq r$ be primes
and let $q_{j}$ be primes such
 that $q_{j} \equiv \pm 2 \pmod{5}$ but $q_{j} \not\equiv \pm7 \pmod{25}$ for
 $1 \leq j \leq s$ with $s \geq 2$. Let $n=p_{1}^{a_{1}} \cdots p_{r}^{a_{r}}q_{1}^{b_{1}}\cdots q_{s}^{b_{s}}$,
 where $1 \leq a_{i},b_{j} \leq 4$ for $1 \leq i \leq r$ and $1 \leq j \leq s$. Let $n \equiv \pm 1, \pm 7 \pmod{25}$.
  Let $F=\Q(\zeta_{5})$ and $K=F(n^{\frac{1}{5}})$. Assume that all ambiguous ideal classes of $K/F$ are strongly ambiguous.
  Then, the $\lambda^2$-rank of $S_{K}$ is $r+s-2$ and $2r+2s-4 \leq \mathrm{rank } S_{K} \leq 4r+4s-8$.\\
If there are ambiguous ideal classes which are not strongly
ambiguous, then $s_{1} \leq 1$, $\lambda^{2}$-rank of $S_{K}$ is
greater than or equal to $r+s-3$ and $\max (2r+2s-5,r+s-2) \leq
\mathrm{rank } S_{K} \leq 4r+4s-8$.
\end{thm}

\begin{proof}
Firstly, we notice that $\lambda$ does not ramify in $K/F$. Since $N(q) \not\equiv 1 \pmod{25}$,
$\zeta \notin N_{K/F}(K^{\ast})$.
Looking at the fields $K_{i}=F(p_{i}^{\frac{1}{5}})$ and $L_{1}=F(q_{1}^{\frac{1}{5}})$,
$L_{j}=F((q_{1}q_{j}^{h_{j}})^{\frac{1}{5}}$, where $1\leq h_{j} \leq 4$ are chosen such that
$q_{1}q_{j}^{h_{j}} \equiv \pm 1,\pm 7 \pmod{25}$, $j \neq 1$, one can easily see that
some $\zeta^{u}(1+\zeta)^{v}$ is a fifth power modulo $p_{i}$ for all $i=1,\cdots,r$ and a fifth power modulo $q_{j}$ for $1 \leq j \leq s$. Thus $q^{\ast}=q=1$ and $t=d-3+q^{\ast}=r+s-2$.\\
To compute $s_{1}$, let $x_{i}=p_{i}$ where $1 \leq i \leq r$ and
$y_{j-1} = q_{1}q_{j}^{h_{j}}$ where $2 \leq j \leq s-1$.
Using \cite{Serre}[Chapter 14, Section 3]  one can easily check that
$\Big(\frac{x_i,n}{p_{j}}\Big) = 1$ for $1 \leq i,j \leq r$, $\Big(\frac{x_i,n}{q_{j}}\Big) = 1$ for $1 \leq i \leq r$,
$1\leq j \leq s$, $\Big(\frac{y_i,n}{p_{j}}\Big) = 1$ for $1 \leq i \leq s-2$, $1\leq j \leq r$ and
$\Big(\frac{y_i,n}{q_{j}}\Big) = 1$ for $1 \leq i \leq s-2$, $1\leq j \leq s$.
That is, the $(r +s-2) \times (r+s)$ matrix $C_{1}$ is the zero matrix. So $s_{1}=0$.\\
We obtain that the $\lambda^2$-rank of $S_{K}$ is $t-s_{1}=r+s-2$. Since $2t-s_{1} \leq \mathrm{rank }S_{K} \leq 4t-3s_{1}$, we get $2r+2s-4 \leq \mathrm{rank } S_{K} \leq 4r+4s-8$.\\
the second part of the statement follows from the fact that $q=0$,
and then the matrix $C_{1}$ is of size $(r+s-2) \times (r+s+1)$,
which can have rank at most $1$.
\end{proof}

\noindent The following table provided by SAGE (which is valid under
GRH) gives the computation of various class groups. We have $F=
\Q(\zeta_{5})$, $K= F(n^{\frac{1}{5}})$. We denote by $n(f)$ the
number of distinct prime divisors of $n$ in $F$ and $S_{K}$ the
$5$-class group of $K$ respectively.

\begin{table}[ht!]
\caption{\bf{Number Fields and their Class Groups}}
\centering
\begin{tabular}{ c  c   c  }
\hline 
\toprule
  n & n(f) & $S_{K}$ \\ \midrule
2,3,4,7,8,9,16,17,23,27 & 1 & 1 \\
43,47,49,53,73,81,97 & 1 & 1 \\
13,37,67,83 & 1& 1 \\ \midrule

18, 24, 26,51,68,74 & 2 & 1 \\

6,12,14,21,28,36,39,48,52 & 2 & $C_{5}$ \\
54,56,69,72,91,92,94,98 & 2 & $C_{5}$ \\
34,46,63,86 & 2 & $C_{5}$ \\
301 & 2 & $C_{5}$ \\ \midrule

19,29,59,79,89 & 2 & $C_{5} \times C_{5}$ \\
57,76 & 3 & $C_{5} \times C_{5}$ \\
38,58,87,133 & 3 & $C_{5} \times C_{5} \times C_{5}$\\ \midrule

42,78,84 & 3 & $ C_{5} \times C_{5} \times C_{5} \times C_{5} \times C_{5}$\\ \midrule

11, 41,61,71 & 4 & $C_{5} \times C_{5}$ \\
31 & 4 & $ C_{5} \times C_{5} \times C_{5} \times C_{5} \times C_{5}$\\
82,93,99 & 5 & $C_{5} \times C_{5}$ \\
22,44,62,77 & 5 & $C_{5} \times C_{5} \times C_{5}$ \\
33,88 & 5 & $ C_{5} \times C_{5} \times C_{5} \times C_{5}$\\
66 & 6 & $ C_{5} \times C_{5} \times C_{5} \times C_{5} \times C_{5}$\\ \midrule

5,25 & 1 & 1 \\
10,15,20,45,75,80 & 2 & 1 \\
40,50,65,85 & 2 & 1 \\
35 & 2 & $C_{5}$ \\
30,60,70,90 & 3 & $C_{5}$ \\
55,95 & 3 & $C_{5} \times C_{5}$\\

\bottomrule
\end{tabular}
\label{Tab1}
\end{table}

\newpage

\begin{example}
In this example, we look at the more complicated situation when
$\pi_{i}$'s are {\bf not} of the form $a \pmod{ 5 \Z[\zeta]}$ for
some nonzero integer $a$. Let $F=\Q(\zeta)$, and $x=11$. So
$K=F(11^{\frac{1}{5}})$. We have
$$11=u\pi_{1}\pi_{2}\pi_{3}\pi_{4},$$ where $u$ is a unit, $\pi_{1} =
2+ \zeta$, $\pi_{2} = 1+ \zeta - \zeta^2$, $\pi_{3} = 1+ \zeta + 2
\zeta^2$ and $\pi_{4}=1-\zeta+\zeta^3$.\\ We saw that in this
situation (example 5.4) $t=2$ and $q^{\ast}=0$. \\Next we note that,
$x_{1}= \pi_{1}^{2}\pi_{2}^{3}\pi_{3}^{3}\pi_{4}^{2} \equiv 1\pmod{
\lambda^5}$ and $x_{2}= -\pi_{1}^{4}\pi_{2}\pi_{3}\pi_{4}^{4} \equiv
1\pmod{ \lambda^5}$. \\Then $M_{1} = K(x_{1}^{\frac{1}{5}},
x_{2}^{\frac{1}{5}})$ is the genus field. \\
Next we note that
$\Big(\frac{x_1,\lambda}{(\lambda)}\Big)=\Big(\frac{x_2,\lambda}{(\lambda)}\Big)
=\zeta^{4}$( \cite{Cassels}[Exercise 2.12, pg.353-354]). \\
We also see that, $\Big(\frac{x_1,x}{(\pi_{1})}\Big) \neq 1$ (it
equals $4 \in F_{11}$)(see \cite{Serre}[Chapter 14, Section 3]). One
then easily see that the $2\times 5$ matrix of theorem 5.7 has {\rm
rank} $2$. Thus $s_{1}=2$ and $S_{K}$ is an elementary abelian 5
group of {\rm rank} 2, that is $S_{K} = C_{5} \times C_{5}$. This is
also confirmed by a SAGE program - see the table.

\end{example}
\vskip 5mm

\noindent We observe from table 1 that, if $p \equiv -1 \pmod{5}$,
then rank of class group is at least 2. That motivated us to prove
the following result (see also table 2 below). The table is obtained
using SAGE under the assumption GRH, with
$K=\Q(\zeta_{5})(p^{\frac{1}{5}})$, where $p \equiv -1 \pmod{5}$ and
$R=\Z[\zeta_{5}]$. The third column describes the $R$-module
structure of the $5$-class group $S_{K}$.

\begin{table}[ht!]
\caption{\bf{Structure of $5$ Class Group of $K$}} \centering
\begin{tabular}{ c  c   c  }
\hline 
\toprule $p$ & $S_{K}$ & Structure of $S_{K}$ as $R$ module \\
\midrule
19 & $C_{5} \times C_{5}$ & $R/(\lambda^2)$ \\
29 & $C_{5} \times C_{5}$ &$ R/(\lambda^2) $\\
59 & $C_{5} \times C_{5} $& $R/(\lambda^2) $\\
79 & $C_{5} \times C_{5} $& $R/(\lambda^2) $\\
89 & $C_{5} \times C_{5} $& $R/(\lambda^2) $\\
109 & $C_{5} \times C_{5} $& $R/(\lambda^2) $\\
139 & $C_{5} \times C_{5} \times C_{5}$& $R/(\lambda^3) $\\
149 & $C_{5} \times C_{5} $& $R/(\lambda^2) $\\
179 & $C_{5} \times C_{5} $& $R/(\lambda^2) $\\
199 & $C_{5} \times C_{5} $& $R/(\lambda^2) $\\
229 & $C_{5} \times C_{5} $& $R/(\lambda^2)$ \\
239 & $C_{5} \times C_{5} $& $R/(\lambda^2) $\\
269 & $C_{5} \times C_{5} $& $R/(\lambda^2)$ \\
349 & $C_{5} \times C_{5}$ & $R/(\lambda^2)$ \\
\bottomrule
\end{tabular}
\label{Tab2}
\end{table}

\newpage

\begin{thm}\label{-1mod5}
Let $p$ be a prime congruent to $-1 \pmod{5}$. Let $F=\Q(\zeta_{5})$
and $K=F(p^{\frac{1}{5}})$. Assume that all ambiguous ideal classes
are strongly ambiguous. Then $25$ divides the class number of $K$.
More precisely, the $\lambda^2$-rank of $S_{K}$ is $1$ and we have,
$ 2 \leq \mathrm{rank } S_{K} \leq 4$.
\end{thm}

\begin{proof}
It is known that any prime of the form $p \equiv -1 \pmod{5}$ can be written as
$$p = a^2+ab-b^2$$
with $a,b \in \Z$, non-zero with $(a,b)=1$. Note that this implies that $(a,p)=(b,p)=1$.\\
Let $c=a-b$, define
$$\pi_{1}= a \zeta^3 + a \zeta^2 +b \text{\quad     and  \quad   } \pi_{2}= a \zeta^3 + a \zeta^2 +c .$$
Now we observe the following two identities:
$$ a^2+bc = a^2+b(a-b) =p \text{\quad and \quad}a^2-ab-ac=a^2 -a(b+c) =0.$$
Thus,
\begin{align*}
\pi_{1}\pi_{2} &= (a \zeta^3 + a \zeta^2 +b)(a \zeta^3 + a \zeta^2 +c) \\
                        & =  (2a^2+bc) + a^2 (\zeta + \zeta^4) + (ab+ac) (\zeta^2+\zeta^3)\\
                        &= (a^2+bc)+(a^2-ab-ac)(1+\zeta+\zeta^4)\\
                        &= p.
\end{align*}
This gives us prime decomposition of $p$ in $F$. Now to compute the $\lambda^2$-rank of $S_{K}$, we compute $t$ and $s_{1}$.\\
If $p \equiv -1 \pmod{25}$, then $N(\pi_{i})=p^2 \equiv 1 \pmod{25}$ for $i=1,2$. So $\zeta \in N_{K/F}(K^{\ast})$.\\
If $p \not\equiv -1 \pmod{25}$, then $N(\pi_{i})=p^2 \not\equiv 1 \pmod{25}$ for $i=1,2$. So $\zeta \notin N_{K/F}(K^{\ast})$.\\
In both cases,
$$ -\zeta^2(1+\zeta) \equiv \frac{b}{a} \pmod{\pi_{1}} \text{ \quad and \quad } -\zeta^2(1+\zeta) \equiv \frac{c}{a} \pmod{\pi_{2}}. $$
Note that $\frac{b}{a}$ and $\frac{c}{a}$ are in $\mathbb{F}_{p}^{\ast}$. Since $5 \nmid p-1$, $x \mapsto x^5$ is an isomorphism of $\mathbb{F}_{p}^{\ast}$. Hence $\frac{b}{a}$ and $\frac{c}{a}$ are fifth power modulo $\pi_{1}$ and $\pi_{2}$ respectively. Thus in both cases we see that, $-\zeta^2(1+\zeta) \in N_{K/F}(K^{\ast})$. \\
Combining these facts, we see that in both cases, $t=g-1=1$.\\
In the case, $p \not\equiv -1 \pmod{25}$, we have $q^{\ast}=q=1$. In the case $p \equiv -1 \pmod{25}$, we have $q^{\ast}=q=2$ and $\lambda$ does not ramify in $K/F$. \\
Let $M_{1}$ denote the genus field of $K/F$. It is of the form
$M_{1}=K(x_{1}^{\frac{1}{5}})$.
Since $M_{1}$ is unramified over $K$, only the primes that ramify in $K$ can divide $x_{1}$.\\
Suppose $p \equiv -1 \pmod{25}$. Then, only $\pi_{1}$ and $\pi_{2}$
ramify in $K$. So $x_{1}$ is of the form $x_{1}=
\pi_{1}^{\alpha_{1}} \pi_{2}^{\alpha_{2}}$. To compute $s_{1}$, we
need to compute $\Big(\frac{x_1,p}{\pi_{1}}\Big)$,
$\Big(\frac{x_1,p}{\pi_{2}}\Big)$. Let $c_{1} = (-1)^{\alpha_{1}}
\frac{x_{1}}{p^{\alpha_{1}}} = (-1)^{\alpha_{1}}
\pi_{2}^{\alpha_{2}-\alpha_{1}}$. Since $\bar{\pi_{2}} = c-b
\pmod{\pi_1}$, with $c-b \in \mathbb{F}_{p}^{\ast}$, we see on using
\cite{Serre}[Chapter 14, Section 3] that,
$\Big(\frac{x_1,p}{\pi_{1}}\Big)=(\bar{c_{1}})^{(p^2-1)/5}=1$. This
was because the residue field is $\mathbb{F}_{p^2}$ and $(p^2-1)/5$
is a multiple of $p-1$. Similarly we find that,
$\Big(\frac{x_1,p}{\pi_{2}}\Big)=1$.
So in this case the $(1 \times 2)$ matrix $C_{1}$ is the zero matrix. Hence $s_{1}=0$.\\
Now suppose the $p\not\equiv -1 \pmod{25}$, then $\lambda$ also
ramifies in $K$. So in this case, $x_{1}$ is of the form, $x_{1}=
\lambda^{a}\pi_{1}^{\alpha_{1}}\pi_{2}^{\alpha_{2}}$. To compute
$s_{1}$, we need to compute $\Big(\frac{x_1,p}{\pi_{1}}\Big)$,
$\Big(\frac{x_1,p}{\pi_{2}}\Big)$ and
$\Big(\frac{x_1,\lambda}{(\lambda)}\Big)$. Let
$c_{1}=(-1)^{\alpha_{1}} \frac{x_{1}}{p^{\alpha_{1}}} =
(-1)^{\alpha_{1}} \pi_{2}^{\alpha_{2}-\alpha_{1}} \lambda^{a}$.
Since $\bar{\pi_{2}} = c-b \pmod{\pi_1}$, with $c-b \in
\mathbb{F}_{p}^{\ast}$ and $(\bar{\lambda})^{4}=5 \pmod{\pi_{1}}$,
we see that $(\bar{c_{1}})^{4(p-1)} = 1$. That is
$\Big(\frac{x_1,p}{\pi_{1}}\Big)=(\bar{c_{1}})^{(p^2-1)/5}=\pm 1$.
Since it is a fifth root of unity in $\mathbb{F}_{p}[\zeta]^{\ast}$,
it can not be $-1$. Thus, $\Big(\frac{x_1,p}{\pi_{1}}\Big)= 1$.
Similarly we see that, $\Big(\frac{x_1,p}{\pi_{2}}\Big)=1$. To
compute $\Big(\frac{x_1,\lambda}{(\lambda)}\Big)$, we compute
$\Big(\frac{x_1,\lambda}{(\pi_{1})}\Big)$ and
$\Big(\frac{x_1,\lambda}{(\pi_{2})}\Big)$ and use the product
formula. We compute $\Big(\frac{x_1,\lambda}{(\pi_{i})}\Big)$
similarly to get, $\Big(\frac{x_1,\lambda}{(\pi_{i})}\Big)=1$ for
$i=1,2$. Thus $\Big(\frac{x_1,\lambda}{(\lambda)}\Big)=1$. So in
this case the $1 \times 3$
matrix $C_{1}$ is the zero matrix. Hence $s_{1}=0$.\\
Thus $s_{1}=0$ in either case which means that the $\lambda^2$-rank
of $S_{K}$ is $t-s_{1}=1$. Lastly, we observe that,
$$2=2t-s_{1} \leq \mathrm{rank } S_{K} \leq 4t-3s_{1} =4,$$
that is 25 divides the class number of $K$.
\end{proof}

\noindent The following theorem is similar in flavor to that of
Theorem \ref{-1mod5} (see also table 3 below). The table is obtained
using SAGE under the assumption GRH, with
$K=\Q(\zeta_{5})(n^{\frac{1}{5}})$, where $n=pq$ with $p\equiv \pm 7
\pmod{25}$, $q \equiv -1 \pmod{5}$ and $R=\Z[\zeta_{5}]$. The third
column describes the $R$-module structure of the $5$-class group
$S_{K}$.
\begin{table}[ht!]
\caption{\bf{Structure of $5$-Class Group of $K$}} \centering
\begin{tabular}{ c  c   c  }
\hline 
\toprule $n=p \times q$ & $S_{K}$ & Structure of $S_{K}$ as $R$
module \\ \midrule
$7 \times 19$ & $C_{5} \times C_{5} \times C_{5}$ & $R/ (\lambda) \times R/(\lambda^2)$\\
$7 \times 29$ & $C_{5} \times C_{5} \times C_{5} \times C_{5}$ & $R/ (\lambda) \times R/(\lambda^3)$\\
$7 \times 59$ & $C_{5} \times C_{5} \times C_{5}$ & $R/ (\lambda) \times R/(\lambda^2)$\\
$7 \times 79$ & $C_{5} \times C_{5} \times C_{5} \times C_{5}$ & $R/ (\lambda) \times R/(\lambda^3)$\\
$7 \times 89$ & $C_{5} \times C_{5} \times C_{5}$ & $R/ (\lambda)
\times R/(\lambda^2)$\\ \midrule
$7 \times 149$ & $C_{5} \times C_{5} \times C_{5}$ & $R/ (\lambda) \times R/(\lambda^2)$\\
$43 \times 149$ & $C_{5} \times C_{5} \times C_{5} \times C_{5}$ & $R/ (\lambda) \times R/(\lambda^3)$\\
$107 \times 149$ & $C_{5} \times C_{5} \times C_{5}$ & $R/ (\lambda) \times R/(\lambda^2)$\\
$7 \times 199$ & $C_{5} \times C_{5} \times C_{5}$ & $R/ (\lambda) \times R/(\lambda^2)$\\
$43 \times 199$ & $C_{5} \times C_{5} \times C_{5}$ & $R/ (\lambda) \times R/(\lambda^2)$\\
$107 \times 199$ & $C_{5} \times C_{5} \times C_{5}$ & $R/ (\lambda) \times R/(\lambda^2)$\\
\bottomrule
\end{tabular}
\label{Tab3}
\end{table}

\newpage

\begin{thm}\label{2primes}
Let $p$ be a prime congruent to $\pm7 \pmod{25}$ and $q$ be a prime
congruent to $-1 \pmod{5}$. Let $F=\Q(\zeta_{5})$ and
$K=F((pq)^{\frac{1}{5}})$. Assume that all ambiguous ideal classes
are strongly ambiguous. Then, $125$ divides the class number of $K$.
More precisely, the $\lambda^2$-rank of $S_{K}$ is $1$, and $3 \leq
\mathrm{rank } S_{K} \leq 5$.
\end{thm}

\begin{proof}
Suppose firstly that $q \equiv -1 \pmod{25}$. Then, $q$ factors as $\pi_{1}\pi_{2}$ in $F$ as in Theorem \ref{-1mod5}. Since $p \equiv \pm 7 \pmod{25}$, $p$ is prime in $F$. We have $N(p)=p^4 \equiv 1 \pmod{25}$ and $N(\pi_{i})=q^2 \equiv 1 \pmod{25}$ for $i=1,2$. Thus $\zeta \in N_{K/F}(K^{\ast})$.\\
As in Theorem \ref{-1mod5}, we see that $1+\zeta$ is a fifth power modulo $\pi_{1}$ and modulo $\pi_{2}$.
Considering the intermediate field $K_{1}=F(p^{\frac{1}{5}})$, we see that the Hasse formula from section 5.1 for this situation gives
$t_{1}=d_{1}-3+q_{1}^{\ast}=q_{1}^{\ast}-2$. Thus $q_{1}^{\ast}=2$, that is, $1+\zeta$ is a fifth power modulo $p$. Thus $1+\zeta \in N_{K/F}(K^{\ast})$. We note that in this case $\lambda$ does not ramify.\\
If $q \not\equiv  -1 \pmod{25}$, then $N(\pi_{i})=q^2 \not\equiv 1 \pmod{25}$ for $i=1,2$. But in this situation, $-\zeta^{2}(1+\zeta) \in N_{K/F}(K^{\ast})$. So $q^{\ast}=1$ and $\lambda$ ramifies.
Combining these facts, we immediately see that, in both cases, $t=g-1=3-1=2$.\\
Next we want to compute $s_{1}$.  Let $x_{1}=p$ and $x_{2}=\pi_{1}^{\alpha_{1}} \pi_{2}^{\alpha_{2}} $ as in the proof of Theorem \ref{-1mod5}. When $q \equiv -1 \pmod{25}$, to compute the matrix $C_{1}$, we need to compute the Hilbert symbols, $\Big(\frac{x_1,pq}{p}\Big)$, $\Big(\frac{x_1,pq}{\pi_{1}}\Big)$, $\Big(\frac{x_1,pq}{\pi_{2}}\Big)$, $\Big(\frac{x_2,pq}{p}\Big)$, $\Big(\frac{x_2,pq}{\pi_{1}}\Big)$, $\Big(\frac{x_2,pq}{\pi_{2}}\Big)$. Using the formula given in \cite{Serre}[Chapter 14, Section 3] , we can see as in Theorem \ref{-1mod5}, that \\
$$\Big(\frac{x_1,pq}{p}\Big)=\Big(\frac{x_1,pq}{\pi_{1}}\Big)=\Big(\frac{x_1,pq}{\pi_{2}}\Big)=\Big(\frac{x_2,pq}{\pi_{1}}\Big)=\Big(\frac{x_2,pq}{\pi_{2}}\Big)=1,$$
and $\Big(\frac{x_2,pq}{p}\Big) \neq 1$. Thus, the $2 \times 3$ matrix $C_{1}$ has only one nonzero entry. So $s_{1}=1$. \\
When $q \not\equiv -1 \pmod{25}$, $C_{1}$ has one more column consisting of $\Big(\frac{x_1,\lambda}{(\lambda)}\Big)$ and
$\Big(\frac{x_2,\lambda}{(\lambda)}\Big)$. We see as in \ref{-1mod5},
$\Big(\frac{x_2,\lambda}{(\lambda)}\Big)=1$ and since $x_{1}= \pm 7 \pmod{25}$, $\Big(\frac{x_1,\lambda}{(\lambda)}\Big)=1$
as well. Then. the $2\times4$ matrix $C_{1}$ in this case also has only one nonzero entry. So $s_{1}=1$. \\
Thus we see that in both cases, the $\lambda^2$-rank of $S_{K}$ is
$t-s_{1}=1$. Lastly, we observe that,
$$3=2t-s_{1} \leq \mathrm{rank } S_{K} \leq 4t-3s_{1} =5.$$
\end{proof}

\noindent {\bf Remarks.} If there are ambiguous ideal classes which
are not strongly ambiguous, then in Theorem \ref{-1mod5}, $s_{1}$
can possibly be equal to $1$; in that case, the rank of $S_{K}$
would be $1$. Similarly, in Theorem \ref{2primes}, $s_{1}$ can
possibly be $2$; in that case, the rank of $S_{K}$ would be $2$.
But, we have not been able to find any example for either of these
situations; perhaps, under the hypotheses of theorem 5.16 or of
theorem 5.17, all ambiguous ideal classes are strongly ambiguous.
\vskip 5mm

\section{$5$-class group of pure quintic fields}

\noindent In this final section, we apply the results of the last
section (especially theorems 5.12,5.13 and 5.14) to deduce results
on some quintic extensions of $\mathbb{Q}$. Let $L$ be a degree $5$
extension of $\Q$ such that $[L(\zeta_{5}):L]=4$ and
$\mathrm{Gal}(L(\zeta_{5})/L) \cong \Z/4\Z=G$. Let $K=L(\zeta_{5})$.

\begin{lem}\label{factors}
Let $C$ be a $\Z_{5}[G]$ module and $G=<\sigma>$. Let $C^{+}= \{ a
\in C | \sigma a=a \}, C^{-}= \{ a \in C |  \sigma a=-a \}$ and
$C^{--}= \{ a \in C | \sigma^{2} a=-a \}$. Then $C \cong C^{+}
\oplus C^{-} \oplus C^{--}$.
\end{lem}

\begin{proof}
Let $a \in C$. Write $a =
({\frac{1+\sigma+\sigma^2+\sigma^3}{4}}a)+(
{\frac{1-\sigma+\sigma^2-\sigma^3}{4}}a
)+({\frac{1-\sigma^2}{2}}a)$. Then
${\frac{1+\sigma+\sigma^2+\sigma^3}{4}}a  \in C^{+}$,
${\frac{1-\sigma+\sigma^2-\sigma^3}{4}}a  \in C^{-}$ and
${\frac{1-\sigma^2}{2}}a  \in C^{--}$. Let $b \in C^{+} \cap C^{-}$,
then  $b = {\sigma}b=-b$, that is $2b=0$. Thus $b=0$ as $C$ is a
$\Z_{5}$ module. Similarly we can show that, $C^{+} \cap
C^{--}=C^{-} \cap C^{--}= \{0\}$. Hence, $C \cong C^{+} \oplus C^{-}
\oplus C^{--}$.
\end{proof}

\begin{lem}\label{L&K}
Let $S_{K}$ and $S_{L}$ denotes the $5$-class group of $K$ and $L$
respectively. Then, $S_{L} \cong S_{K}^{+}$ and $S_{L}/5S_{L} \cong
(S_{K}/5S_{K})^{+}$.
\end{lem}

\begin{proof}
We have a natural inclusion $S_{L} \hookrightarrow S_{K}$ as $5$ is
relatively prime to $[K:L]=4$. Moreover, $S_{L} \hookrightarrow
S_{K}^{+}$ as ${\sigma}a=a$ for all $a \in S_{L}$. Let $a \in
S_{K}^{+}$, then $a = 4({\frac{1}{4}}a) = (1+\sigma
+\sigma^{2}+\sigma^{3})({\frac{1}{4}}a)= N({\frac{1}{4}}a)$. Thus,
$a \in S_{L}$. So $S_{L} \cong S_{K}^{+}$. Now, $S_{L}/5S_{L} \cong
S_{K}^{+}/ 5(S_{K}^{+}) \cong (S_{K}/5S_{K})^{+}$.
\end{proof}

\subsection{Decomposing $S_K$ under the affine group of
$\mathbb{F}_5$}

\noindent Now let $L$ is a pure quintic field, that is
$L=\Q(n^{\frac{1}{5}})$, where $n$ is a positive integer which does
not contain any $5^{th}$ power. Let $F=\Q(\zeta_{5})$ and
$K=F(n^{\frac{1}{5}})=L(\zeta_{5})$. Then $K$ is a cyclic extension
of degree $5$ over $F$ and we can use the theory developed in the
previous sections to determine $S_{K}$. Let $\sigma$ be a generator
of $G=\mathrm{Gal}(K/L)$ and $\tau$ be a generator of
$\mathrm{Gal}(K/F)$. We observe that, $K/\Q$ is Galois. We fix the
generators $\sigma, \tau$ in $\mathrm{Gal}(K/\Q)$ satisfying the
relations
$$\sigma^{4}=\tau^5=1~,~\sigma \tau = \tau^{3} \sigma.$$
Let $\lambda = 1 - \tau$.\\
Note that $S_{K}, 5S_{K}$ are $\Z_{5}[G]$-modules. Consider the
filtration
$$S_{K} \supset \lambda S_{K} \supset \lambda^{2}S_{K} \supset \lambda^{3}S_{K} \supset 5 S_{K}=\lambda^{4} S_{K}.$$
Using the relations $\sigma \tau = \tau^3 \sigma$ and $\tau = 1-
\lambda$, we note that
$$\sigma \lambda a = \lambda((\lambda^{2}-3
\lambda+3) \sigma a),$$ $$ \sigma \lambda^{2} a =
(\lambda^{2}((2\lambda - 1)\sigma a))5(\lambda(1-\lambda)\sigma
a),$$ $$ \sigma \lambda^{3} a = (\lambda^{3}((2-\lambda)\sigma
a))5(\lambda(1-\lambda)(2-\lambda) \sigma a).$$ So,
$\lambda^{i}S_{K}$ for $0 \leq i \leq 4$ are $\Z_{5}[G]$-modules.
Using lemma \ref{factors}, we get,
$$\lambda^{i}S_{K}/5S_{K} \cong (\lambda^{i}S_{K}/5S_{K})^{+} \oplus (\lambda^{i}S_{K}/5S_{K})^{-} \oplus (\lambda^{i}S_{K}/5S_{K})^{--}  \text{   for } 0\leq i \leq 3,$$
$$\lambda^{i}S_{K}/\lambda^{i+1}S_{K} \cong (\lambda^{i}S_{K}/\lambda^{i+1}S_{K})^{+} \oplus (\lambda^{i}S_{K}/\lambda^{i+1}S_{K})^{-} \oplus (\lambda^{i}S_{K}/\lambda^{i+1}S_{K})^{--} \text{   for } 0\leq i \leq 3.$$
The natural projection $\lambda^{i}S_{K}/5S_{K} \to
\lambda^{i}S_{K}/\lambda^{i+1}S_{K}$ is surjective with kernel
$\lambda^{i+1}S_{K}/5S_{K}$. Restricting to the + part, we get the
surjective map $(\lambda^{i}S_{K}/5S_{K})^{+} \to
(\lambda^{i}S_{K}/\lambda^{i+1}S_{K})^{+}$ with kernel
$(\lambda^{i+1}S_{K}/5S_{K})^{+}$. Since $(S_{K}/5S_{K})^{+},
(\lambda^{i}S_{K}/\lambda^{i+1}S_{K})^{+}$ are of exponent $5$, we
have,
$$
\mathrm{rank} S_{L} = \mathrm{rank} (S_{K}/5S_{K})^{+}$$
$$ = \mathrm{rank} (S_{K}/\lambda S_{K})^{+} +\mathrm{rank} (\lambda S_{K}/\lambda^{2}S_{K})^{+}
$$
$$+ \mathrm{rank} (\lambda^{2}S_{K}/\lambda^{3}S_{K})^{+} +
\mathrm{rank} (\lambda^{3}S_{K}/\lambda^{4}S_{K})^{+}.$$

\noindent The rank of $(S_{K}/\lambda S_{K})^{+}$ can be read off from the generators of the genus field,
which we will describe at the end. We first determine the rank of $(\lambda^{i}S_{K}/\lambda^{i+1}S_{K})^{+}$ for $i \geq 1$.\\
As before, consider the map induced by multiplication by $\lambda$:
$$\lambda_{i}^{\ast}: \lambda^{i}S_{K}/\lambda^{i+1}S_{K} \to \lambda^{i+1}S_{K}/\lambda^{i+2}S_{K}$$
$$  a \pmod{\lambda^{i+1}S_{K}} \mapsto \lambda a \pmod{\lambda^{i+2}S_{K}}.$$
We observe that,
$${\sigma^{2}} \lambda a \equiv -\lambda \sigma^{2}a \pmod{\lambda^{2}S_{K}}.$$
Thus $$\lambda_{0}^{\ast} ( (S_{K}/\lambda S_{K})^{+} \oplus
(S_{K}/\lambda S_{K})^{-}) \subset (\lambda
S_{K}/\lambda^{2}S_{K})^{--}$$ and  $$\lambda_{0}^{\ast} (
(S_{K}/\lambda S_{K})^{--}) \subset (\lambda
S_{K}/\lambda^{2}S_{K})^{+} \oplus (\lambda
S_{K}/\lambda^{2}S_{K})^{-}.$$ Thus we have two surjective maps,
$$\theta_{1}: (S_{K}/\lambda S_{K})^{+} \oplus (S_{K}/\lambda S_{K})^{-} \to (\lambda S_{K}/ \lambda^{2}S_{K})^{--}, $$
$$ \theta_{2}: (S_{K}/\lambda S_{K})^{--} \to (\lambda S_{K}/\lambda^{2}S_{K})^{+} \oplus (\lambda S_{K}/\lambda^{2}S_{K})^{-}.$$
We note that $\mathrm{rank}~\mathrm{Ker} \theta_{1} + \mathrm{rank}~\mathrm{Ker} \theta_{2} = s_{1}$. \\
We have,
$$\mathrm{rank} (\lambda S_{K}/\lambda^{2} S_{K})^{+} = \mathrm{rank} (S_{K}/\lambda S_{K})^{--} - \mathrm{rank}
(\lambda S_{K}/\lambda^{2}S_{K})^{-} - \mathrm{rank}~\mathrm{Ker}
\theta_{2}.$$

\noindent Next we observe that
$$\sigma \lambda^{2} a \equiv - \lambda^{2} \sigma a \pmod{\lambda^{3}S_{K}}~,~
\sigma^{2}\lambda^{2}a \equiv \lambda^{2}\sigma^{2}a
\pmod{\lambda^{3}S_{K}} .$$ Thus, $$\lambda_{1}^{\ast} ( (\lambda
S_{K}/\lambda S_{K})^{+}) \subset
(\lambda^{2}S_{K}/\lambda^{3}S_{K})^{-},$$ $$\lambda_{1}^{\ast} (
(\lambda S_{K}/\lambda^{2}S_{K})^{-}) \subset
(\lambda^{2}S_{K}/\lambda^{3}S_{K})^{+},$$
$$\lambda_{1}^{\ast} (
(\lambda S_{K}/\lambda^{2}S_{K})^{--}) \subset
(\lambda^{2}S_{K}/\lambda^{3}S_{K})^{--}.$$ We thus obtain three
surjective maps,
$$\alpha_{1}: (\lambda S_{K}/\lambda^{2} S_{K})^{-} \to (\lambda^{2}S_{K}/\lambda^{3}S_{K})^{+},$$
$$\alpha_{2}: (\lambda S_{K}/\lambda^{2} S_{K})^{+} \to (\lambda^{2}S_{K}/\lambda^{3}S_{K})^{-},$$
$$\alpha_{3}: (\lambda S_{K}/\lambda^{2} S_{K})^{--} \to (\lambda^{2}S_{K}/\lambda^{3}S_{K})^{--}.$$
We note that $\mathrm{rank}~\mathrm{Ker} \alpha_{1} +
\mathrm{rank}~\mathrm{Ker} \alpha_{2} + \mathrm{rank}~\mathrm{Ker}
\alpha_{3} = \mathrm{rank} \frac{(\lambda S_K))[\lambda] + \lambda^2
S_K}{\lambda^2 S_K} = s_{2}$. \\
So, we have
$$\mathrm{rank} (\lambda^{2}S_{K}/\lambda^{3}S_{K})^{+} = \mathrm{rank}
(\lambda S_{K}/\lambda^{2} S_{K})^{-} - \mathrm{rank } \mathrm{Ker} \alpha_{1}.$$

\noindent Finally we observe that,
$$\sigma^{2} \lambda^{3} a \equiv -\lambda^{3} \sigma^{2}a \pmod{\lambda^{4}S_{K}}.$$
Thus, we obtain
$$\lambda_{2}^{\ast} (
(\lambda^{2}S_{K}/\lambda^{3}S_{K})^{+} \oplus
(\lambda^{2}S_{K}/\lambda^{3}S_{K})^{-}) \subset
(\lambda^{3}S_{K}/\lambda^{4}S_{K})^{--}$$ and
$$\lambda_{2}^{\ast}
( (\lambda^{2}S_{K}/\lambda^{3}S_{K})^{--}) \subset
(\lambda^{3}S_{K}/\lambda^{4}S_{K})^{+} \oplus
(\lambda^{3}S_{K}/\lambda^{4}S_{K})^{-}.$$ Therefore, we have two
surjective maps,
$$\beta_{1}: (\lambda^{2}S_{K}/\lambda^{3}S_{K})^{+} \oplus (\lambda^{2}S_{K}/\lambda^{3}S_{K})^{-} \to (\lambda^{3}S_{K}/\lambda^{4}S_{K})^{--}, $$
$$ \beta_{2}: (\lambda^{2}S_{K}/\lambda^{3}S_{K})^{--} \to (\lambda^{3}S_{K}/\lambda^{4}S_{K})^{+} \oplus (\lambda^{3}S_{K}/\lambda^{4}S_{K})^{-}.$$
We note that $\mathrm{rank}~\mathrm{Ker} \beta_{1} +
\mathrm{rank}~\mathrm{Ker} \beta_{2} = \mathrm{rank}
\frac{(\lambda^2
S_K))[\lambda] + \lambda^3 S_K}{\lambda^3 S_K} = s_{3}$. \\
Notice that
$$\mathrm{rank} (\lambda^{3}S_{K}/\lambda^{4}S_{K})^{+} = \mathrm{rank} (\lambda^{2}S_{K}/\lambda^{3}S_{K})^{--} - \mathrm{rank} (\lambda^{3}S_{K}/\lambda^{4}S_{K})^{-} - \mathrm{rank}~\mathrm{Ker} \beta_{2}.$$
It is easy to see that,
\begin{multline*}
\mathrm{rank} (\lambda^{3}S_{K}/\lambda^{4}S_{K})^{+} = \mathrm{rank} (S_{K}/\lambda S_{K})^{+} +  \mathrm{rank} (S_{K}/\lambda S_{K})^{-} - \mathrm{rank}~\mathrm{Ker} \theta_{1}\\
 - \mathrm{rank}~\mathrm{Ker} \alpha_{3} - \mathrm{rank}~\mathrm{Ker} \beta_{2} - \mathrm{rank} (\lambda^{3}S_{K}/\lambda^{4}S_{K})^{-}.
\end{multline*}
Putting these together, we get
\begin{multline*}
\mathrm{rank} S_{L} =  \mathrm{rank} (S_{K}/\lambda S_{K})^{+} +  \mathrm{rank} (S_{K}/\lambda S_{K}) - (\mathrm{rank}~\mathrm{Ker} \theta_{1}+\mathrm{rank}~\mathrm{Ker} \theta_{2} \\
 + \mathrm{rank}~\mathrm{Ker} \alpha_{1} + \mathrm{rank}~\mathrm{Ker} \alpha_{3} + \mathrm{rank}~\mathrm{Ker} \beta_{2} + \mathrm{rank} (\lambda^{3}S_{K}/\lambda^{4}S_{K})^{-}).
 \end{multline*}
Noting that, $\mathrm{rank} (S_{K}/\lambda S_{K})=t$ and
$\mathrm{rank}~\mathrm{Ker} \theta_{1}+\mathrm{rank}~\mathrm{Ker}
\theta_{2}=s_{1}$, we obtain an upper bound for the rank of $S_{L}$
as
$$ \mathrm{rank} S_{L} \leq t-s_{1}+  \mathrm{rank} (S_{K}/\lambda S_{K})^{+} \leq 2t-s_{1}.$$
Thus, we have proved the following theorem.

\begin{thm}\label{ranksl}
Let $L=\Q(n^{\frac{1}{5}})$, where $n$ is an integer which does not
contain any fifth power. Let $F=\Q(\zeta_{5})$ and
$K=F(n^{\frac{1}{5}})=L(\zeta_{5})$. Then,
$$ \mathrm{rank } S_{L} \leq \mathrm{rank }(S_{K}/\lambda S_{K})^{+} + \lambda^2-{\rm rank} S_K = \mathrm{rank }(S_{K}/\lambda S_{K})^{+} + (t-s_{1}) \leq 2t-s_{1}.$$
\end{thm}

\begin{cor}\label{t=s1}
When $t=s_{1}$, $\mathrm{rank} S_{L} =\mathrm{rank }(S_{K}/\lambda S_{K})^{+}$.
\end{cor}

\begin{proof}
Since $t=s_{1}$, $\lambda S_{K}=5 S_{K}$ and $s_{2}=s_{3}=0$. Thus $\mathrm{Ker} \alpha_{i} =\mathrm{Ker} \beta_{j}=0 $ for $i=1,2,3$ and $j=1,2$. Moreover, $\lambda^{3}S_{K}/\lambda^{4}S_{K} =0$. So,
$$\mathrm{rank} S_{L} = \mathrm{rank }(S_{K}/\lambda S_{K})^{+}+ t-s_{1}=\mathrm{rank }(S_{K}/\lambda S_{K})^{+}.$$
\end{proof}

\subsection{Kummer duality to bound rank of $(S_K/\lambda S_K)^+$}

\noindent Finally we describe how one can determine $\mathrm{rank
}(S_{K}/\lambda S_{K})^{+}$ or give an upper bound for this rank.
Let $M$ be the maximal abelian unramified extension of $K$ with
exponent $5$. By class field theory, we have, $S_{K}/5S_{K} \cong
\mathrm{Gal}(M/K)$. By Kummer theory there exists a subgroup $A$ of
$K^{\ast}$,
$$(K^{\ast})^{5} \subset A \subset K^{\ast},$$
such that $M=K(\sqrt[5]{A})$. We have a bilinear pairing
$$A/(K^{\ast})^{5} \times \mathrm{Gal}(M/K) \to \{ 5^{th}\text{ roots of unity} \}$$
$$ (x,\mu) \mapsto [x, \mu] = (x^{\frac{1}{5}})^{\mu -1}.$$
By Kummer theory $A/(K^{\ast})^{5}$ and $\mathrm{Gal}(M/K)$ are dual
groups with respect to this pairing. Thus identifying $S_{K}/5S_{K}$
with $\mathrm{Gal}(M/K)$ we see that $A/(K^{\ast})^{5}$ and
$S_{K}/5S_{K}$ are dual groups in the bilinear pairing. Let $M_{1}$
be a field $K \subset M_{1} \subset M$ and $M_{1}/K$ is Galois. By
Kummer theory, there is a subgroup $B$ of $A$ such that
$$(K^{\ast})^{5} \subset B \subset A \subset K^{\ast}$$ and
$M_{1}=K(\sqrt[5]{B})$. Moreover, there is a group $T$, satisfying
$5S_{K} \subset T \subset S_{K}$, such that $S_{K}/T$ is dual of
$B/(K^{\ast})^{5}$ and $S_{K}/T \cong \mathrm{Gal}(M_{1}/K)$. One
can easily check that $[x^{\sigma}, \mu^{\sigma}]
=[x,\mu]^{\sigma}$, where $\sigma$ is the generator of
$\mathrm{Gal}(K/L)$, $x \in B/(K^{\ast})^{5}$ and $\mu \in S_{K}/T$,
$\mu^{\sigma} = z_{\sigma}^{-1} \mu z_{\sigma}$ where $z_{\sigma}
\in \mathrm{Gal}(M_{1}/L)$ is a element which maps to $\sigma$ under
the natural projection. \\
Writing $$(B/(K^{\ast})^{5})^{+} = \{ b \in B/(K^{\ast} | b^{\sigma}
= b \},$$ $$(B/(K^{\ast})^{5})^{-} = \{ b \in B/(K^{\ast} |
b^{\sigma} = b^{-1} \},$$ $$ (B/(K^{\ast})^{5})^{--} = \{ b \in
B/(K^{\ast} | b^{\sigma^{2}} = b^{-1} \},$$ we have the following
lemma:

\begin{lem}
Let
$$B/(K^{\ast})^{5} \times S_{K}/T \to \{ 5^{th} \text{ roots of unity in} K \}$$
$$(x,\mu) \mapsto [x,\mu]$$
be the bilinear pairing described above.  Then $(B/(K^{\ast})^{5})^{\dagger}$ is orthogonal to $(S_{K}/T)^{\dagger}$ under the pairing, where $\dagger \in \{+,-,--\}$. Moreover, $(B/(K^{\ast})^{5})^{\pm}$ is orthogonal to $(S_{K}/T)^{\mp}$ under the pairing.
\end{lem}

\begin{proof}
Let $x \in (B/(K^{\ast})^{5})^{+}$ and $\mu \in (S_{K}/T)^{+}$. Then $[x,\mu]=[x^{\sigma},\mu^{\sigma}] =
[x,\mu]^{\sigma}=[x,\mu]^{2} \text{ or } [x,\mu]^{3}$, since $\zeta^{\sigma}= \zeta^{2} \text{ or } \zeta^{3}$.
Thus $[x,\mu]=1$ or $[x,\mu]^{2}=1$. In either case, we see that $[x,\mu]=1$. Thus $(B/(K^{\ast})^{5})^{+}$ and
$(S_{K}/T)^{+}$ are orthogonal in this pairing. the other cases are similar.\\
For the second part, let $x \in (B/(K^{\ast})^{5})^{+}$ and $\mu \in (S_{K}/T)^{-}$.
Then $[x,\mu]=[x^{\sigma},(\mu^{-1})^{\sigma}] = [x,\mu^{-1}]^{\sigma}=([x,\mu]^{-1})^{\sigma}=[x,\mu]^{3} \text{ or }
[x,\mu]^{2}$, since $\zeta^{\sigma}= \zeta^{2} \text{ or } \zeta^{3}$. Thus $[x,\mu]^{2}=1$ or $[x,\mu]=1$.
In either case, we see that $[x,\mu]=1$. Thus $(B/(K^{\ast})^{5})^{+}$ and $(S_{K}/T)^{-}$ are orthogonal in this pairing.
The other case is similar.
\end{proof}

\noindent Since $(B/(K^{\ast})^{5})^{-}$ and $(S_{K}/T)^{+}$ are
orthogonal and $(B/(K^{\ast})^{5})^{+}$ and $(S_{K}/T)^{-}$ are
orthogonal, the dual group of $(S_{K}/T)^{+} \oplus (S_{K}/T)^{-}$
is contained in $(B/(K^{\ast})^{5})^{--}$. On the other hand,
$(B/(K^{\ast})^{5})^{--}$ and $(S_{K}/T)^{--}$ are orthogonal, and
hence the dual group of $(B/(K^{\ast})^{5})^{--}$ is contained in
$(S_{K}/T)^{+} \oplus (S_{K}/T)^{-}$. We see therefore that,
$(S_{K}/T)^{+} \oplus (S_{K}/T)^{-}$ is dual to
$(B/(K^{\ast})^{5})^{--}$ under this pairing. Let
$M_{1}=K(x_{1}^{\frac{1}{5}},\cdots,x_{t}^{\frac{1}{5}})$ be the
genus field of $K/F$, that is $S_{K}/\lambda S_{K} \cong
\mathrm{Gal}(M_{1}/K)$. Suppose $x_{1},\cdots,x_{w}$ are the
rational numbers among the $x_i$'s.  Then, $\mathrm{rank}
(B/(K^{\ast})^{5})^{+} =w$.  Suppose $x_{w+1},\cdots,x_{r}$ are the
$x_i$'s whose factors only contain rational numbers and primes of
the form $(a \zeta^{2}+a \zeta^{3}+b)$. Noticing that for an element
$\pi = (a \zeta^{2}+a \zeta^{3}+b)$, we get $\pi^{\sigma^{2}} = \pi
\neq \pi^{\sigma}$, we have $\mathrm{rank} (B/(K^{\ast})^{5})^{-}
=r-w$ and $\mathrm{rank} (B/(K^{\ast})^{5})^{--} = t-r$. Hence
$$\mathrm{rank} (S_{K}/\lambda S_{K})^{+} \leq t-r.$$
In particular, we obtain the theorem:

\begin{thm}\label{r=0}
Let $L=\Q(n^{\frac{1}{5}})$, where $n=p_{1}^{a_{1}} \cdots p_{m}^{a_{m}} q_{1}^{b_{1}} \cdots q_{u}^{b_{u}}$ where $p_{i} \equiv \pm 2 \pmod{5}$, $q_{j} \equiv -1 \pmod{5}$ and $1 \leq a_{i},b_{j} \leq 4$ for $i \in \{ 1,\cdots,m\}$ and for $j \in \{ 1,\cdots,u\}$. Let $F=\Q(\zeta_{5})$ and $K=F(n^{\frac{1}{5}})=L(\zeta_{5})$. Then $\mathrm{rank }(S_{K}/S_{K}^{\Delta})^{+}=0$ and
$$t-s_{1}-s_{2} = \lambda^3-\mathrm{rank}~\mathrm{of} S_K  \leq \mathrm{rank } S_{L} \leq \lambda^2-\mathrm{rank}~\mathrm{of}~~ S_K = t-s_{1}.$$
\end{thm}

\begin{proof}
As proved in proposition \ref{generators} and theorem \ref{-1mod5},
all the generators of the genus field $M_{1}$ are either rational
integers with prime factors $p_{i}$ or contains factors of $q_{j}$
in $F$. But, as we proved in theorem \ref{-1mod5}, factors of
$q_{j}$'s are of the form $(a \zeta^{2}+a \zeta^{3}+b)$. Hence, with
$r$ as defined before the theorem, we have $r=t$. Then, the upper
bound for rank of $S_L$ follows from theorem \ref{ranksl}. For the
lower bound, notice that, $(S_{K}/\lambda S_{K})^{+} \oplus
(S_{K}/\lambda S_{K})^{-}$ is dual to $(B/(K^{\ast})^{5})^{--}$
which is $0$. We have,
$$(\lambda S_{K}/\lambda^{2} S_{K})^{--}=(\lambda^{2}
S_{K}/\lambda^{3} S_{K})^{--} =(\lambda^{3} S_{K}/\lambda^{4}
S_{K})^{+}=(\lambda^{3} S_{K}/\lambda^{4} S_{K})^{-}=0$$ and
consequently, the maps $\alpha_{3}, \beta_{2}$ defined in section
5.1 are both $0$. Thus
$${\rm rank } S_{L}= {\rm rank }(S_{K}/\lambda S_{K})^{+} + t -s_{1} - {\rm rank} \text{ }{\rm Ker} \alpha_{1}
= (t-s_{1}-s_{2})+ {\rm rank} \text{ } {\rm Ker} \alpha_{2} \geq t-s_{1}-s_{2}.$$
\end{proof}

\noindent {\bf Remarks:} \\
(i) We would have better results if we can get more information
about rank $(S_K/\lambda S_K)^+$.\\
(ii) Under the assumption that all ambiguous ideal classes are
strongly ambiguous, we computed the $\lambda^2$-rank of $S_{K}$ in
theorems \ref{-7mod25},\ref{not-7mod25},\ref{is-7mod25},\ref{-1mod5}
and \ref{2primes}. If there are ambiguous ideal classes which are
not strongly ambiguous, the maximum value of $s_{1}$ is given.

\begin{cor}
Let $L=\Q(N^{\frac{1}{5}})$. In the following cases $S_{L}$ is
trivial or cyclic.
\begin{itemize}
\item Let $N=p^{a}$, where $p \equiv \pm2 \pmod{5}$ is a prime, $1 \leq a \leq 4$.
\item Let $N=q_{1}^{a_{1}}q_{2}^{a_{2}}$ where  $q_{i} \equiv \pm 2 \pmod{5}$ but $q_{i} \not\equiv \pm 7 \pmod{25}$, $1 \leq a_{i} \leq 4$ for $i=1,2$ such that $N \equiv \pm 1, \pm 7 \pmod{25}$.
\item Let $N=p^{a}$, where $p \equiv -1 \pmod{5}$ is a prime, $1 \leq a \leq 4$.
\item Let $N=p_{1}^{a_{1}}p_{2}^{a_{2}}$ where $p_{i} \equiv \pm 7 \pmod{25}$, $1 \leq a_{i} \leq 4$ for $i=1,2$ such that $N \equiv \pm 1, \pm 7 \pmod{25}$.
\item $N=p^{a}q^{b}$ where $p \equiv \pm 7 \pmod{25}$ , $q \equiv \pm 2 \pmod{5}$ but $q \not\equiv \pm 7 \pmod{25}$ and $1 \leq a,b \leq 4$such that $N \not\equiv \pm 1, \pm 7 \pmod{25}$.
\item $N=q_{1}^{a_{1}}q_{2}^{a_{2}}$ where  $q_{i} \equiv \pm 2 \pmod{5}$ but $q_{i} \not\equiv \pm 7 \pmod{25}$, $1 \leq a_{i} \leq 4$ for $i=1,2$ such that $N \not\equiv \pm 1, \pm 7 \pmod{25}$.
\item  $N=p_{1}^{a_{1}}p_{2}^{a_{2}}q^{b}$ where $p_{i} \equiv \pm 7 \pmod{25}$, $q \equiv \pm 2 \pmod{5}$ but $q \not\equiv \pm 7 \pmod{25}$  $1 \leq a_{i},b \leq 4$ for $i=1,2$ such that $N \equiv \pm 1,\pm 7 \pmod{25}$.
\item $N=p^{a}q_{1}^{a_{1}}q_{2}^{a_{2}}$ where $p \equiv \pm 7 \pmod{25}$, $q_{i} \equiv \pm 2 \pmod{5}$ but $q_{i} \not\equiv \pm 7 \pmod{25}$, $1 \leq a,a_{i} \leq 4$ for $i=1,2$ such that $N \equiv \pm 1,\pm 7 \pmod{25}$.
\item  $N=q_{1}^{a_{1}}q_{2}^{a_{2}}q_{3}^{a_{3}}$ where  $q_{i} \equiv \pm 2 \pmod{5}$ but $q_{i} \not\equiv \pm 7 \pmod{25}$, $1 \leq a_{i} \leq 4$ for $i=1,2,3$, such that $N \equiv \pm 1, \pm 7 \pmod{25}$.
\item Let $N=p^{a}q^{b}$, where $p \equiv -1 \pmod{5}$ and $q \equiv \pm 7 \pmod{25}$ are primes, $1 \leq a,b \leq 4$.
\end{itemize}
\end{cor}

\begin{proof}
In all these situations, $\mathrm{rank }(S_{K}/\lambda S_{K})^{+}=0$. From Theorem \ref{-7mod25},\ref{not-7mod25},\ref{is-7mod25},\ref{-1mod5} follows that for each of these cases except the last case $t=0 \text{ or }1$. In the last case, we see from Theorem \ref{2primes} that $t=2$ and $s_{1} \geq 1$. Thus $\mathrm{rank} S_{L} \leq 1$. The result follows. Note that, in first two cases the class groups are trivial.
\end{proof}

\noindent {\bf Remark:} Let $f$ be a normalized cuspidal Hecke
eigenform of weight $k$ and level $N$. Let $K_{f}$ denote the
extension of $\Q_{5}$ generated by the $q$-expansion coefficients
$a_{n}(f)$ of $f$. It is known that $K_{f}$ is a finite extension of
$\Q_{5}$. In the case $N$ is prime and $5 || N-1$, it is known
\cite{Mazur} that there exists unique (up to conjugation) weight $2$
normalized cuspidal Hecke eigenform defined over $\QQ_{5}$,
satisfying the congruence
$$a_{l}(f) \equiv 1+l \pmod{\mathfrak{p}}$$
where $\mathfrak{p}$ is the maximal ideal of the ring of integer of
$K_{f}$, and $l \neq N$ are primes. In this situation it is also
known that $K_{f}$ is a totally ramified extension of $\Q_{5}$ and
let $[K_{f}:\Q_{5}] = e_{5}$. Calegari and Emerton \cite{Calemer}
showed that $e_{5}=1$ if the class group of $\Q(N^{\frac{1}{5}})$ is
cyclic. They also showed that if $N \equiv 1 \pmod{5}$, and the
$5$-class group of $\Q(N^{\frac{1}{5}})$ is cyclic, then
$\prod_{l=1}^{(N-1)/2} l^{l}$ is not a $5^{th}$ power modulo $N$.
This corollary gives us information when $5$-class group of
$\Q(N^{\frac{1}{5}})$ is cyclic for various $N$.

\newpage

\begin{table}[ht!]
\caption{\bf{Structure of $5$ Class Group of $L$}}
\centering
\begin{tabular}{ c  c }
\hline 
\toprule
$n$ & $S_{L}$ \\ \midrule
$2 \times 7$ & $1$\\
$3 \times 7$ & $1$\\
$7 \times 43$ & $1$\\
$2 \times 3 \times 7$ & $C_{5}$\\
$2 \times 13 \times 7$ & $C_{5}$\\
$7 \times 107$ & $C_{5}$\\
$3 \times 13 \times 7$ & $C_{5} \times C_{5}$\\ \midrule

$19,29,59,79,89,109,139,149,179,199,229,239,269,349$ & $C_{5}$ \\ \midrule

$7 \times 19, 7 \times 29, 7 \times 59,7 \times 79,7 \times 89, 7 \times 149, 7 \times 199$ & $C_{5}$\\
 $43 \times 19, 43 \times 29, 43 \times 59,43 \times 79,43 \times 89, 43 \times 149, 43 \times 199$ & $C_{5}$\\ \midrule

$11,41,61,71,101,151,191,241,251,271$ & $C_{5}$ \\
$31,131,181$ & $C_{5} \times C_{5}$\\
$211,281$ & $C_{5} \times C_{5} \times C_{5}$ \\
\bottomrule
\end{tabular}
\label{Tab4}
\end{table}

\noindent
{\bf Acknowledgments.}\\

\noindent It is a pleasure to thank Dipendra Prasad for
clarifications regarding genus theory. We are also indebted to
Suprio Bhar for help with some computer codes in SAGE.  

\vskip 10mm

\noindent {\bf Addresses of authors:} \vskip 5mm

\noindent {\it Manisha Kulkarni}, Department of mathematics,
International Institute of Information Technology, 26/C, Electronics
City, Hosur
Road, Bangalore 560100, India.\\
{\em manisha.shreesh@gmail.com}\\

\noindent {\it Dipramit Majumdar}, Indian Institute of Science Education and Research, Dr. Homi Bhaba Road, Pashan, Pune 411008, India.\\
{\em dipramit@gmail.com}\\

\noindent {\it Balasuramanian Sury}, Stat-Math Unit, Indian
Statistical Institute, 8th Mile
Mysore Road, Bangalore 560059, India.\\
{\em surybang@gmail.com}

\end{document}